\documentclass[10pt, twoside, fleqn, a4paper]{article}
\usepackage{amsmath}
\usepackage{amssymb}
\usepackage{amsthm}
\usepackage{amsfonts}
\usepackage{color}
\usepackage{dsfont}
\usepackage{fullpage}
\usepackage{graphicx}
\usepackage{tabularx}
\usepackage[tight]{subfigure}
\usepackage{times}
\usepackage{enumitem}
\usepackage{pdfpages}

\newcommand{\tildeB}{\widetilde{B}}

\newcommand{\tildeD}{\widetilde{D}}

\newcommand{\de}{\delta}

\newcommand{\ga}{\gamma}
\newcommand{\la}{\lambda}

\newcommand{\si}{\sigma} 
\newcommand{\deriv}[1]{\frac{\partial}{\partial #1}}

\newcommand{\E}[1]{\mathbb{E}[#1]}
\newcommand{\Pro}[1]{\mathbb{P}(#1)}
\newcommand{\SCV}[1]{\ensuremath{c^2_{#1}}}

\newcommand{\Cov}[2]{\ensuremath{\mbox{\textup{Cov}}[#1, #2]}}
\newcommand{\Var}[1]{\ensuremath{\mbox{\textup{Var}}[#1]}}

\newcommand{\deq}{\,{\buildrel d \over =}\,}

\newcommand{\figurewidth}{0.5\textwidth}

\renewcommand{\vec}[1]{\ensuremath{\boldsymbol{#1}}} % vector

\theoremstyle{plain}% default
\newtheorem{theorem}{Theorem}[section]
\newtheorem{apprx}[theorem]{Approximation}
\newtheorem{lemma}[theorem]{Lemma}

\theoremstyle{definition}
\newtheorem{remark}{Remark}[section]

\newtheorem*{example}{Example}

\theoremstyle{remark}

\definecolor{Blue}{rgb}{0.3,0.3,0.9}
\definecolor{Red}{cmyk}{0,1.,1.,0}

\newcommand{\footnoteremember}[2]
{
  \newcounter{#1}\footnote{#2}\setcounter{#1}{\value{footnote}}  
}
\newcommand{\footnoterecall}[1]
{
  \footnotemark[\value{#1}] 
} 

\begin{document}
\title{Marginal queue length approximations for a two-layered network with correlated queues\footnotetext{Funded in the framework of the STAR-project ``Multilayered queueing systems'' by the Netherlands Organization for Scientific Research (NWO). The research of M.\ Vlasiou is also partly supported by an NWO individual grant through project 632.003.002.}}
\author{
	J.L. Dorsman\footnoteremember{TU/e}{EURANDOM and Department of Mathematics and Computer Science, Eindhoven University of Technology, P.O. Box 513, 5600 MB Eindhoven, The Netherlands.}\footnoteremember{CWI}{Probability and Stochastic Networks, Centrum Wiskunde \& Informatica (CWI), P.O. Box 94079, 1090 GB Amsterdam, The Netherlands.}\\
	\small \texttt{j.l.dorsman@tue.nl}\\
	\and
	O.J. Boxma\footnoterecall{TU/e}\\
	\small \texttt{o.j.boxma@tue.nl}\\
	\and
	M. Vlasiou\footnoterecall{TU/e}\footnoterecall{CWI}\\
	\small \texttt{m.vlasiou@tue.nl}
}

\date{September 2013}
\maketitle

\begin{abstract}
We consider an extension of the classical machine-repair model, where we assume that the machines, apart from receiving service from the repairman, also serve queues of products. The extended model can be viewed as a layered queueing network, where the first layer consists of the queues of products and the second layer is the ordinary machine-repair model. Since the repair time of one machine may affect the time the other machine is not able to process products, the downtimes of the machines are correlated. This correlation leads to dependence between the queues of products in the first layer. Analysis of these queue length distributions is hard, since the exact dependence structure for the downtimes, or the queue lengths, is not known. Therefore, we obtain an approximation for the complete marginal queue length distribution of any queue in the first layer, by viewing such a queue as a single server queue with correlated server downtimes. Under an explicit assumption on the form of the downtime dependence, we obtain exact results for the queue length distribution for that single server queue. We use these exact results to approximate the machine-repair model. We do so by computing the downtime correlation for the latter model and by subsequently using this information to fine-tune the parameters we introduced to the single server queue. As a result, we immediately obtain an approximation for the queue length distributions of products in the machine-repair model, which we show to be highly accurate by extensive numerical experiments.
\end{abstract}

\section{Introduction}
In this paper, we study a layered queueing network (LQN) consisting of two layers. We define an LQN to be a queueing network where in addition to the traditional ``servers'' and ``customers'', there exist entities that act as servers for upper-layer customers and as customers for lower-layer entities. So far, the study of such networks has been limited to computer-science problems; see \cite{FranksEtAl} and references therein for an overview.

The LQN under consideration is motivated by a two-fold extension of the traditional machine-repair model. This model, also known as the \emph{computer terminal model} (cf. \cite{BertsekasGallager}) or as the \emph{time sharing system} (cf. \cite[Section 4.11]{Kleinrock}), is a well-studied problem in the literature. In the machine-repair model, there is a number of machines (two in our case) working in parallel, and one repairman. As soon as a machine fails, it joins a repair queue in order to be repaired by the repairman.
It is one of the key models to describe problems with a finite input population. A fairly extensive analysis of the machine-repair model can be found in Tak\'acs \cite[Chapter 5]{Takacs}. 

We extend this model in two directions. First, we allow machines to have different uptime or repair time distributions. As observed in \cite{GrossInce}, this leads to technical complications. For example, the arrival theorem (cf. \cite{LavenbergReiser}) cannot be used any more to derive the stationary downtime distributions of the machines as is done for the original model in \cite{Wartenhorst}. Secondly, we assume that each of the machines processes a stream of products, which leads to the addition of queues in front of the machines. Observe that in this case a machine has a dual role. As in the traditional model, the machine has a \emph{customer role} with respect to the repairman, but it now also has a \emph{server role} with respect to the products. This leads to a formulation of a LQN with two layers, which we also refer to as the \emph{two-layered model} or simply the \emph{layered model}.

\begin{figure}
\begin{center}
\includegraphics[width=\figurewidth]{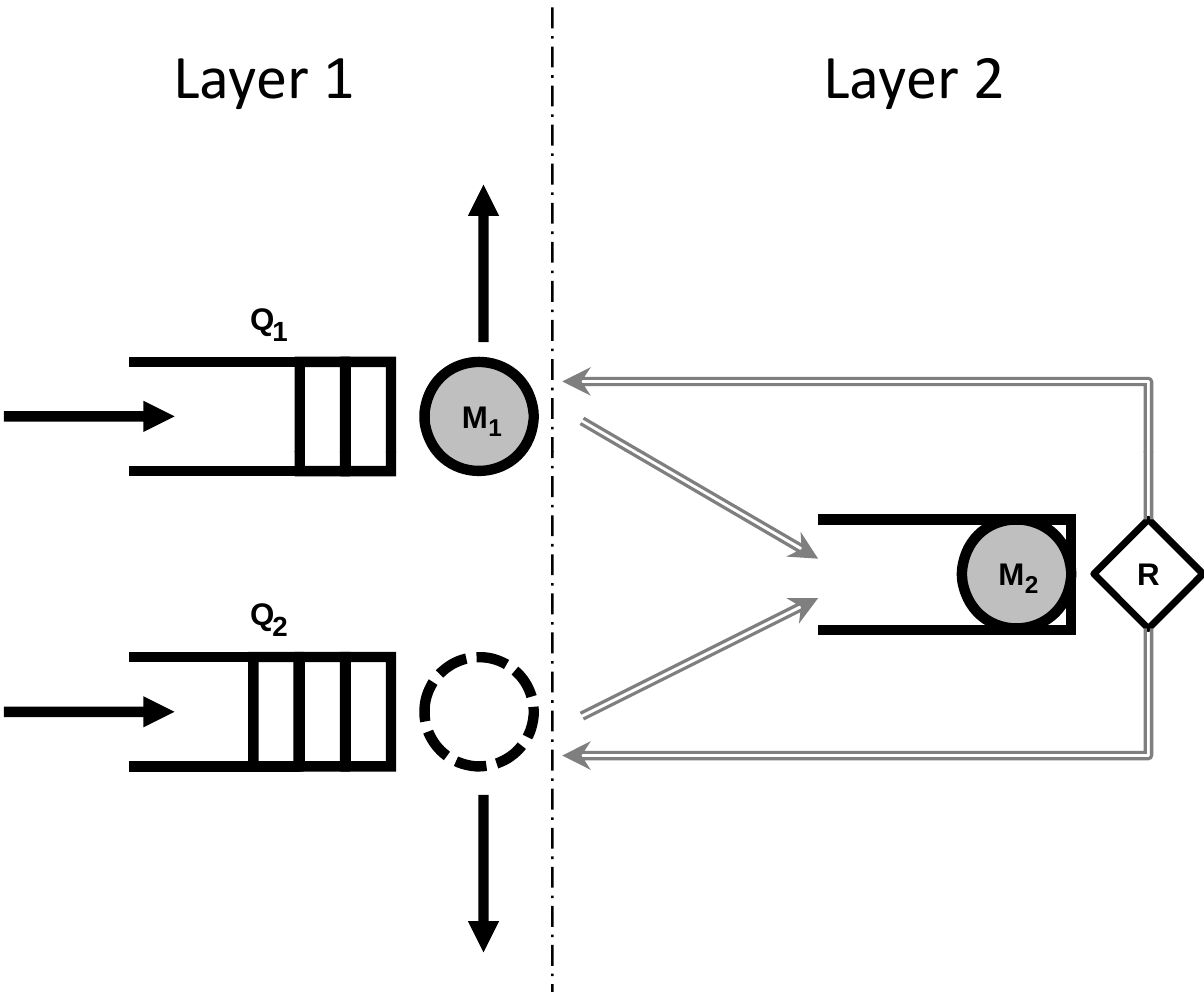}
\caption{The two-layered model under consideration.}
\label{fig:layeredModel}
\end{center}
\end{figure}
The first layer of the resulting LQN contains two queues of products, see Figure \ref{fig:layeredModel}. Each of these queues is served by its own machine. For ease of the discussion, we thus assume that there are two machines only, as opposed to the classical machine-repair model. As will be evident in the sequel, the approach we follow can be readily extended to more machines or repairmen, but computations become increasingly cumbersome. At any point in time, a machine is subject to breakdowns, irrespective of the state of the first-layer queue. When a machine breaks down, the service of a product in progress is aborted. Any progress made is lost, the service requirement resets (\emph{pre-emptive repeat}) and the service starts anew once the machine becomes operational again. 

The second layer consists of a repairman and a repair buffer. If, upon breakdown of a machine, the repairman is idle, the machine is immediately taken into service. Once the machine is again operational, it starts serving products once more. A machine which upon breakdown meets a busy repairman waits in the buffer. As soon as the other machine is repaired, repair of the current machine starts.

In the present study, we analyse the marginal queue length distributions of the queues in the first layer. An important feature of both the classical machine-repair model and the two-layered model under consideration is the fact that machines compete for repair facilities. This introduces significant positive dependencies in their downtimes and thus in the lengths of the queues in the first layer. If the downtime of one machine is very large, the repair time is probably taking longer than usual, increasing the likelihood for the other machine to break down in the meantime. This has an increasing influence on the next downtime of the first machine, leading to positively correlated consecutive downtimes of a machine. As a result, in the two-layered model, a first-layer queue in isolation can be seen as a single server vacation queue with \emph{dependent} server vacation times.  In this paper, we will call this the \emph{single server queue}. Note that the correlation between the machines' downtimes also leads to correlations in the lengths of the first-layer queues. Hence, there is interaction between the layers, which makes analysis of the first layer rather complicated. 

The difficulty one faces when studying this model is that the dependence between the downtimes is not well understood, since it is only implicitly defined through the uptimes and repair times of the machines. Exact results on the queue length of a first-layer queue are not available in general. This holds even for the case of Poisson product arrivals, exponential uptimes and exponential repair times. Although the queue length can then in principle be modelled as a reflected Markov-modulated L\'evy process, its distribution is not easily derived from that. For example, standard results on these processes such as \cite[Proposition XI.2.2]{Asmussen} lead to a linear system with a number of equations larger than the number of unknowns, defying straight-forward solutions. Numerical methods to evaluate the queue lengths are available, such as simulation or the power-series algorithm (see e.g. \cite{BlancOverview} or \cite{DorsmanVDMeiVlasiou}). However, these methods do not give insights into the parameter effects on the queue length distributions, may be cumbersome to implement, need relatively much computation time and do not scale well in e.g. the number of machines or repairmen considered. For these reasons, one has to resort to approximations.

We do so by drawing a connection between the layered model and the single server queue. We model the latter as an M/G/1 queue with dependent successive vacation lengths, in order to capture the correlations of the machines' repair times in the layered model. Thus, we use the following approach to approximate the layered model:
\begin{enumerate}[topsep=0pt, itemsep=0pt, leftmargin=*]
\item For the single server queue, we use an explicit, generic dependence form and obtain exact results for the marginal queue length distributions.
\item For the layered model, we compute several characteristics of the downtime structure, such as the first two moments of the downtime distribution and the auto-correlation coefficient of a machine's consecutive downtimes.
\item We choose the parameters of the generic dependence form of the single server model so that they match the downtime characteristics of the layered model computed in the previous step.
\end{enumerate}
Thus, we use the exact results from step one with the parameters from step two to derive an approximation for the marginal queue length distributions of the first-layer queues in the layered model. As we will see in an extensive numerical study, the resulting approximation performs well over a wide range of parameter settings.

Another approximation for the layered model has been proposed by Wartenhorst \cite{Wartenhorst}, who derives approximations for the first two moments of the queue length distribution of a first-layer queue. In his model he assumes equal uptime and repair time distributions for the machines, an assumption that we generalise in this work. In his study, Wartenhorst approximates the first two moments of a first-layer queue with those of a single server vacation queue, where the distribution of the vacations lengths is taken to be equal to that of the machine's downtimes in the layered model, but the vacation lengths are assumed to be \emph{completely independent}. The resulting approximation is exact by construction for a system where downtimes are independent, and accurate whenever downtimes are only slightly dependent. Since the dependence is completely ignored, Wartenhorst's approximation becomes more inaccurate as the dependence increases. In this paper, we explicitly model the dependence, thus improving accuracy greatly, and obtain an approximation for the \emph{complete distribution} of the queue length.

As mentioned, we draw a connection with the M/G/1 queue with server vacations. This queue has been studied extensively, see e.g.\ \cite{DoshiSurvey, DoshiBoek} for surveys. Often vacation lengths or downtimes are assumed to be independent of any other event in the system. Examples of studies where downtimes are assumed to be dependent can be found in \cite{HarrisMarchal}, where vacation lengths are dependent on the number of customers in the system, and in \cite{BoxmaMandjesKella}, where vacation lengths are dependent on the length of the previous active period of the server. In the context of polling systems, vacation queues with interdependent vacation lengths are considered in \cite{Altman, Eliazar, GroeneveltAltman}. However, in that context, the start of a server vacation is usually confined to a point in time where the server concludes the service of a customer. This is not the case in the current context, where a machine can break down at any point in time.

Section \ref{notation} introduces the notation required for both the two-layered model and the single server queueing model, and describes the dependence form used. In Section \ref{analysisSSQ}, we analyse the queue length distribution of the single server queue at various time epochs. This results in an expression for the (PGF of its) steady state queue length distribution at an arbitrary point in time. We believe this result to be of independent interest, but our main goal is to apply this result to the layered model. The obtained PGF provides an approximation for the marginal queue length distribution of the corresponding first-layer queue in the layered model. This forms the main result of this paper and is discussed in Section \ref{approx}. In Section \ref{numStudy}, extensive numerical results show that the obtained approximation is highly accurate. We further identify the factors determining the level of accuracy. We conclude in Section \ref{conclusion} with a summary and the major conclusions of this work.

\section{Notation}\label{notation}
In this section, we introduce the notation we use and describe the dependence structure we consider. 

\paragraph*{The two-layered model.} 
The layered model consists of two machines $M_1$ and $M_2$ and one repairman $R$, see Figure \ref{fig:layeredModel}. Each machine $M_i$ serves its own first-layer queue $Q_i$ on a first-come-first-serve (FCFS) basis.  Products arrive at $Q_i$ according to a Poisson process with rate $\lambda_i$. The service times required by the products in $Q_i$ are generally distributed according to $B_i$. The Laplace-Stieltjes transform (LST) of the service time, $\E{e^{-sB_i}}$ ($Re(s) \ge 0$), is denoted by $\tildeB_i(s)$. The steady-state queue length of $Q_i$, including the product in service, is denoted by $L_i$. After an exponentially ($\sigma_i$) distributed uptime or lifetime, denoted by $U_i$, a machine $M_i$ will break down, and the service of $Q_i$ inevitably stops. The service of a product in progress is then aborted, and will be restarted once the machine is operational again (\emph{pre-emptive repeat}). When a machine breaks down, it moves to the repair buffer, where it will wait if the repairman is busy repairing the other machine; otherwise the repair will start immediately. A downtime $D_i$ of a machine thus consists of a repair time and possibly a waiting time. The time {$R_i$} needed for a repairman to return $M_i$ to an operational state is {generally} distributed. After a repair, the machine returns to $Q_i$ and commences service again. Finally, analogous to the notion of load defined for the single server model, we define
\begin{equation}
\rho_i := \lambda_i\E{B_i}<\frac{\E{U_i}}{\E{U_i}+\E{D_i}}.
\end{equation}
As noted before, the consecutive downtimes are positively cross-correlated due to the interaction of the machines through the repairman. As a result, from the point of view of the products, the machine can be seen as the server in an M/G/1 queue with one-dependent server vacations, where the products themselves are the customers. This is how we approximate the marginal queue length of the products for the layered model. Details on the basic assumptions are given below.

\paragraph*{Single server model with one-dependent server vacations.}
In the single server model, the queue is fed by a Poisson process with parameter $\lambda$. The service time $B$ required by arriving customers is generally distributed. The uptime $U$ from the moment a server has just ended a vacation period until the start of the next one is exponentially distributed with parameter $\sigma$. After this time period $U$, the server starts a vacation for $D$ time units (a downtime). If a job is in service when the server breaks down, all of the work done on the job is lost and processing of the job is restarted once the server ends its vacation (\emph{pre-emptive repeat}).
The steady-state queue length of the queue, including the job in service, is denoted by $L$. The LST of the service time, $\E{e^{-sB}}$ ($Re(s) \ge 0$), is denoted by $\tildeB(s)$. Likewise, $\tildeD(s)$ represents the LST of $D$. 

This model differs from most vacation queues studied in literature, because the durations of vacations (or breakdowns) here are one-dependent. For the dependence, we assume a generic structure, which can be used to model positive correlations between consecutive downtimes. We describe the dependence structure of the downtimes by specifying the LST of a downtime $D(k+1)$ conditioned on its previous downtime $D(k)$:
\begin{equation}\label{eq:combedecomposition}
\E{e^{-sD(k+1)}|D(k)=t} = \chi(s)e^{-g(s)t}, \qquad Re(s) \ge 0,
\end{equation}
where $\chi(s)$ and $g(s)$ are analytic functions in $s$ with $\chi(0) = 1-g(0) = 1$. This generic dependence structure is introduced in \cite{BoxmaCombe} to model positive correlation between two random variables. From \eqref{eq:combedecomposition}, we have that $D(k+1)$ can be interpreted as follows. It can be seen as a sum of an independent component represented by the LST $\chi(s)$, and a component dependent on the previous downtime, represented by $e^{-g(s)t}$. In particular, if one assumes that $g(s)$ has a completely monotone derivative, i.e.\ $(-1)^{n+1}\frac{d^n}{d s^n}g(s) \ge 0$ for all $n \ge 1$, then $e^{-g(s)}$ is the LST of an infinitely divisible distribution (see \cite[p.\ 450]{Feller}). We will use this assumption in the proof of  Lemma \ref{lem:uniqueRootPhi}. 

To give an indication of how rich the class of dependence structures that satisfy \eqref{eq:combedecomposition} is, note that the class of infinitely divisible distributions is strongly connected to the class of L\'evy processes (see e.g. \cite[Chapter 1]{Kyprianou}). In particular, for a L\'evy process $\{X(t), t \ge 0\}$, one has $\E{e^{-sX(t)}} = e^{-g(s)t}$. Thus, $D(k+1)$ consists of a time independent from the previous downtime $D(k)$, and another component, the value of which is that of a L\'evy process observed at a time which is governed by $D(k)$. For several examples of dependence structures that \eqref{eq:combedecomposition} covers, see e.g. \cite{BoxmaCombe} or \cite{VlasiouAdanBoxma}.

In the layered model, a downtime can also be thought of as a sum of an independent component (e.g., the repair time) and a component dependent on the previous downtime (the waiting time). Therefore, the functions $\chi(s)$ and $g(s)$ can be chosen in such a way that they together represent the distribution and the dependence of these downtimes closely. As we discuss in Section \ref{approx}, \eqref{eq:combedecomposition} does not model the downtimes of the layered model perfectly. However, as we will see in Section \ref{numStudy}, it is a good fit.

Note that the functions $\chi(s)$ and $g(s)$ determine the stationary downtime $D := \lim_{k\rightarrow\infty} D(k)$. In steady-state ($k \rightarrow \infty$) it holds that $\E{e^{-sD(k+1)}} = \E{e^{-sD(k)}} = \tildeD(s)$, so we have that 
\begin{equation}\label{eq:identityD}
\tildeD(s) = \int_{t=0}^\infty \chi(s)e^{-g(s)t} d\Pro{D<t} = \chi(s)\tildeD(g(s)),
\end{equation}
and thus
\begin{equation}\label{eq:momentsD}
\E{D} = -\tildeD'(0) = \frac{\chi'(0)}{g'(0)-1} \; \mbox{and} \; \E{D^2} = \tildeD''(0) = \frac{\chi''(0)-\E{D}(2\chi'(0)g'(0)+g''(0))}{1-g'(0)^2}.
\end{equation}
By iterating \eqref{eq:combedecomposition}, one obtains an explicit expression for $\tildeD(s)$:
\begin{equation}\label{eq:infprodform}
\tildeD(s) = \prod_{j=0}^\infty \chi(g^{(j)}(s)),
\end{equation}
where $g^{(0)}(s) = s$ and $g^{(j)}(s) = g(g^{(j-1)}(s))$. The bivariate LST of $D(k)$ and $D(k+1)$ is given by
\begin{equation}\label{eq:jointDLST}
\E{e^{-sD(k)-zD(k+1)}} = \int_{t=0}^\infty e^{-st} \E{e^{-zD(k+1)}|D(k) = t} d \Pro{D(k)<t}= \chi(z)\E{e^{-(s+g(z))D(k)}},
\end{equation}
out of which the joint expectation of two subsequent downtimes $D(k)$ and $D(k+1)$ can be derived:
\begin{equation}\label{eq:jointDExpectation}
\E{D(k)D(k+1)} = \deriv{s}\deriv{z}\chi(z)\E{e^{-(s+g(z))D(k)}}|_{s=0, z=0} = -\chi'(0)\E{D(k)} + g'(0)\E{D(k)^2}.
\end{equation}
We obtain an expression for the bivariate LST in steady-state (i.e. $k \rightarrow \infty$) by combining \eqref{eq:infprodform} and \eqref{eq:jointDLST}:
\begin{equation*}
\lim_{k \rightarrow \infty} \E{e^{-sD(k)-zD(k+1)}} = \chi(z)\prod_{j=0}^\infty \chi(g^{(j)}(s+g(z))).
\end{equation*}
Finally, the stability condition for the single server model is given by 
\begin{equation}
\rho := \lambda\E{B} < \frac{\E{U}}{\E{U}+\E{D}}.
\end{equation}

\section{Analysis of the single server model}\label{analysisSSQ}
In this section, we compute the exact (PGF of the) queue length distribution of the single server model with one-dependent vacations. We later use these results to derive approximations for the layered model. 

We first derive an expression for the PGF of $N$, the queue length distribution at the beginning of an uptime, by studying the transient behaviour of the queue for two server up-down cycles. Thus, we obtain PGF's for $M$ and ultimately $L$, the queue lengths at the end of an uptime and at an arbitrary point in time. An observation length of one cycle would not suffice, since we explicitly need to take the dependence between consecutive downtimes (and thus dependence between cycle lengths) into account. Thus, we observe the system in its $k^{th}$ uptime $U(k)$, as well as the following $k^{th}$ downtime $D(k)$ and in the periods $U(k+1)$ and $D(k+1)$ thereafter. Referring to the queue length distribution at the end of an uptime as $M$, let $N(k)$, $M(k)$, $N(k+1)$, $M(k+1)$ be the corresponding queue lengths, see Figure \ref{fig:queueLength}.
\begin{figure}
\setlength{\unitlength}{2.7cm}
\begin{picture}(1,1)
\put(0,0){\line(1,0){0.6}}
\put(0.6,0){\line(0,1){0.75}}
\put(0.6,0.75){\line(1,0){1.0}}
\put(0.4, 0.85){${N(k)}$}
\put(1.0, 0.2){$U(k)$}
\put(1.4, 0.85){${M(k)}$}
\put(1.6,0.75){\line(0,-1){0.75}}
\put(1.6,0){\line(1,0){0.6}}
\put(1.725, 0.2){$ D(k)$}
\put(2.2,0){\line(0,1){0.75}}
\put(2.2,0.75){\line(1,0){1.3}}
\put(1.9, 0.85){$N(k+1)$}
\put(2.55, 0.2){$U(k+1)$}
\put(3.2, 0.85){$ M(k+1)$}
\put(3.5,0.75){\line(0,-1){0.75}}
\put(3.5,0){\line(1,0){1}}
\put(3.7, 0.2){$ D(k+1)$}
\put(4.5,0){\line(0,1){0.75}}
\put(4.5,0.75){\line(1,0){1}}
\put(4.2, 0.85){$ N(k+2)$}
\end{picture}
\caption{Two server up/down cycles.}
\label{fig:queueLength}
\end{figure}
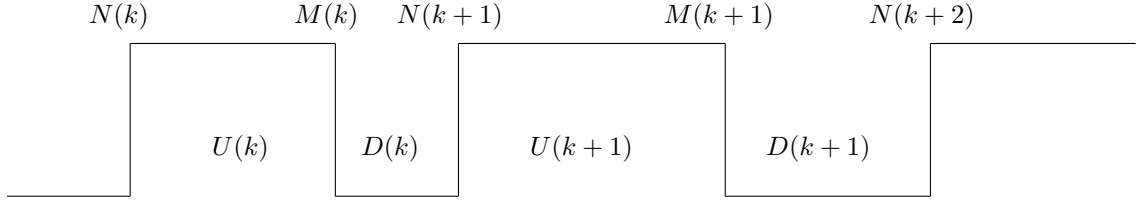
For $k \rightarrow \infty$, we obviously have that
\begin{equation}\label{eq:id} 
\E{p^{N(k)}} = \E{p^{N(k+2)}} = \E{p^N}.
\end{equation}
In Section \ref{NkPlusTwoInNk}, we first express $\E{p^{N(k+2)}}$ in terms of $\E{p^{N(k)}}$. This leads to an expression for $\E{p^{M(k)}}$ in $\E{p^{N(k)}}$, $\E{p^{N(k+1)}}$ in $\E{p^{M(k)}}$, etc. We then compute $\E{p^N}$ in Section \ref{finalisingEPN}.  In Section \ref{obtainingEPL}, we use the results for the embedded times to compute $\E{p^M}$ and $\E{p^L}$, the PGF of the queue length at an arbitrary point in time. We conclude the analysis of the single server model in Section \ref{effectDependence} by illustrating the effects of dependence in downtimes. We believe that the analysis of such a single server queue with dependence between successive vacations is not only useful for studying the layered model, but is also of independent interest.

\subsection{Behaviour of the queue length in two server up/down cycles}\label{NkPlusTwoInNk}
To obtain a relation between $\E{p^{N(k+2)}}$ and $\E{p^{N(k)}}$, we observe the way the queue length evolves in each of the periods $U(k)$, $D(k)$, $U(k+1)$ and $D(k+1)$. Connecting the results leads to an expression for $\E{p^{N(k+2)}}$ in terms of  $\E{p^{N(k)}}$.

\subsubsection{The queue length distribution during the first uptime}
We first derive a relation between $\E{p^{M(k)}}$ and $\E{p^{N(k)}}$. During the first uptime $U(k)$, the server is accepting and processing customers. This means that the queue length in this period of time evolves similarly to the length of a regular M/G/1 queue during an exponential ($\sigma$) interval. This M/G/1 queue has the same customer arrival process and the same service time distribution, but does not have any service interruptions or server downtimes.

A relation between the PGFs of the queue length distribution at the beginning and the end of an exponentially distributed time interval in an M/G/1 queue can be obtained from the queue length transition probabilities between these two points in time. In \cite[p.\ 246]{CohenSingleServer}, these transition probabilities are derived as well as the resulting relation between the queue lengths at the beginning and end of an exponentially distributed time interval. The relation between $M(k)$ and $N(k)$ in our context immediately follows:

\begin{equation} \label{eq:periodUK}
\E{p^{M(k)}} = A(p)*\E{p^{N(k)}}+K(p)*\E{\mu^{N(k)}(\si)},
\end{equation}
where 
\begin{align*}
{A(p)} & {=\frac{\si}{\si+\la(1-p)}\frac{p(1-\tildeB(\si+\la(1-p)))}{p-\tildeB(\si+\la(1-p))}}, \\
{K(p)} & {=-\frac{\si}{\si+\la(1-\mu(\si))}\frac{(1-p)\tildeB(\si+\la(1-p))}{p-\tildeB(\si+\la(1-p))}} 
\end{align*}
and $\mu(\sigma)$ is the LST of a busy period in the regular M/G/1 queue evaluated at $\sigma$.  The value $\mu(\si)$ is the unique root of the expression $p - \tildeB(\si+\lambda(1-p))$  with $|\mu(\si)| < 1$ (for a proof of uniqueness, see \cite[p.\ 47--49]{Takacs}). Therefore, $\mu(\si)$ is a pole of both $A(p)$ and $K(p)$, but these poles compensate each other. More specifically, we find by standard methods the following result that we will need in the sequel:
\begin{align}
\lim_{p\rightarrow \mu(\si)}& \big[A(p) + K(p)\big]\notag  \\
=& \lim_{p\rightarrow \mu(\si)} \Big(\frac{\si}{\si+\la(1-p)}+\left(\frac{\si}{\si+\la(1-p)}-\frac{\si}{\si+\la(1-\mu(\si))}\right)\frac{(1-p)\tildeB(\si+\la(1-p))}{p-\tildeB(\si+\la(1-p))}\Big) \notag \\
=&  \frac{\si}{\si+\la(1-\mu(\si))}+\frac{\la\mu(\si)\si(1-\mu(\si))}{\big(1+\la B'(\si+\la(1-\mu(\si)))\big)\big(\si+\la(1-\mu(\si)))^2}. \label{eq:limAmuPlusBmu}
\end{align}

\subsubsection{The queue length distribution during the first downtime}\label{periodDK}
During the first downtime $D(k)$, the server does not process any customers. Therefore, the queue length increases by the number of customer arrivals in this period. More specifically, the difference between $M(k)$ and $N(k+1)$ is exactly the number of Poisson arrivals during $D(k)$. It will prove convenient in later calculations to condition on the event $D(k) = t$ for any $t \in \mathbb{R}_+$. Let $H(t)$ be Poisson ($\lambda t$) distributed, i.e.\ the number of Poisson arrivals during $D(k)=t$. We then obtain the following relation between $\E{p^{N(k+1)}|D(k) = t}$ and $\E{p^{M(k)}}$:
\begin{align}
\E{p^{N(k+1)}|D(k) = t} =& \E{p^{M(k)+H(t)}} \notag \\
							 =& \E{p^{M(k)}}\sum_{i=0}^\infty p^i e^{-\la t}\frac{(\la t)^i}{i!} \notag \\
							 =& \E{p^{M(k)}}e^{-\la(1-p)t}.\label{eq:periodDK} 
\end{align}

\subsubsection{The queue length distribution during the second uptime}
We now obtain a relation between $\E{p^{M(k+1)}|D(k)=t}$ and $\E{p^{N(k+1)}|D(k)=t}$. During the second uptime $U(k+1)$, the server is processing customers for an exponentially ($\sigma$) distributed amount of time, which means that the analysis is largely the same as the analysis of the queue length during the first uptime $U(k)$. The only difference stems from the fact that we now choose to condition on the event $D(k)=t$, in order to be able to concatenate all the results later on. Analogous to \eqref{eq:periodUK}, we have
\begin{equation} \label{eq:periodUKplusOne}
\E{p^{M(k+1)}|D(k)=t} = A(p)*\E{p^{N(k+1)}|D(k)=t}+K(p)*\E{\mu^{N(k+1)}(\si)|D(k)=t}
\end{equation}
with $A(p)$, $K(p)$ and $\mu(\si)$ as before.

\subsubsection{The queue length distribution during the second downtime}\label{periodDKplusOne}
To obtain a relation between $\E{p^{N(k+2)}|D(k)=t}$ and $\E{p^{M(k+1)}|D(k)=t}$, note that the server is not processing customers during the period $D(k+1)$, which again means that the difference between $M(k+1)$ and $N(k+2)$ is equal to the number of Poisson arrivals during the period $D(k+1)$. The duration of $D(k+1)$ is dependent on $D(k)$, which is described by the LST in \eqref{eq:combedecomposition} conditioned on $D(k)=t$. Therefore, the previously introduced conditioning on the event $D(k)=t$ for $t \in \mathbb{R}$ is convenient at this point. Extending the analysis of Section \ref{periodDK} to the second downtime, conditioned on the duration of the first downtime, we implement the dependence in \eqref{eq:combedecomposition} and obtain the following relation:
\begin{align}
\E{p^{N(k+2)}|D(k)=t} &= \int_{u=0}^\infty \E{p^{M(k+1)+H(u)}|D(k)=t} \;d\Pro{D(k+1)<u|D(k)=t} \notag \\
							 &= \int_{u=0}^\infty \E{p^{M(k+1)}|D(k)=t}e^{-\la(1-p)u} \;d\Pro{D(k+1)<u|D(k)=t} \notag \\
							 &= \E{p^{M(k+1)}|D(k)=t}\E{e^{-\la(1-p)D(k+1)}|D(k)=t} \notag \\
							 &= \E{p^{M(k+1)}|D(k)=t}\chi(\la(1-p))e^{-g(\la(1-p))t}, \label{eq:periodDKplusOne}
\end{align}
where $\E{p^{H(u)}} = e^{-\la(1-p)u}$ is the PGF of the number of Poisson arrivals during a time period $u$.

\subsubsection{Connecting all periods}
Combining \eqref{eq:periodUK}, \eqref{eq:periodDK}, \eqref{eq:periodUKplusOne} and \eqref{eq:periodDKplusOne},
we obtain an expression for $\E{p^{N(k+2)}|D(k)=t}$ in terms of $\E{p^{N(k)}}$. Keeping in mind \eqref{eq:limAmuPlusBmu} and the fact that $\mu(\si)$ is a pole of $A(p)$ and $K(p)$, we note that for the substitution of $\E{\mu^{M(k)}(\si)}$, the following important observation holds:
\begin{align*}
\lim_{p \rightarrow \mu(\si)}\E{p^{M(k)}}=& \lim_{p\rightarrow\mu(\si)} \sum_{i=0}^\infty \big(A(p)*p^i + K(p) * \mu^i(\si)\big)\Pro{N(k) = i} \\
=& \sum_{i=0}^\infty\Big( \lim_{p\rightarrow \mu(\si)} \big[A(p) + K(p)\big]\mu^i(\si) + \lim_{p\rightarrow \mu(\si)} \big[A(p)(p^i-\mu^i(\si))\big]\Big)\Pro{N(k) = i} \\
=& \sum_{i=0}^\infty\Big( \lim_{p\rightarrow \mu(\si)} \big[A(p) + K(p)\big]\mu^i(\si) \\
&\qquad \qquad+ \frac{\si (1-\mu(\si))}{(\si+\la(1-\mu(\si)))(1+\la \tildeB'(\si+\la(1-\mu(\si))))}i\mu^i\Big)\Pro{N(k) = i} \\
=& \Big(\lim_{p\rightarrow \mu(\si)} \big[A(p) + K(p)\big]\Big)\E{\mu^{N(k)}(\si)} \\
&\qquad \qquad+ \frac{\si (1-\mu(\si))}{(\si+\la(1-\mu(\si)))(1+\la \tildeB'(\si+\la(1-\mu(\si))))}\E{N(k)\mu^{N(k)}(\si)}. 
\end{align*}
Hence, an extra term containing $\E{N(k)\mu^{N(k)}(\si)}$ arises in the expression for $\E{p^{N(k+2)}|D(k)=t}$. We obtain
\begin{align}
&\E{p^{N(k+2)}|D(k) = t} =\chi(\la(1-p))A^2(p)e^{-(\la(1-p)+g(\la(1-p)))t}\E{p^{N(k)}} \notag \\
&+ \chi(\la(1-p))K(p)\Big(A(p)e^{-(g(\la(1-p))+\la(1-p))t}\notag \\ 
&\qquad\qquad\qquad\qquad\qquad +\lim_{p\rightarrow\mu(\si)}\big[A(p)+K(p)\big]e^{-(g(\la(1-p))+\la(1-\mu(\si)))t}\Big)\E{\mu^{N(k)}(\si)}  \notag \\
&+\chi(\la(1-p))K(p)e^{-(g(\la(1-p))+\la(1-\mu(\si)))t} \notag \\
&\qquad\qquad\qquad\qquad\qquad * \frac{\si (1-\mu(\si))}{(\si+\la(1-\mu(\si)))(1+\la \tildeB'(\si+\la(1-\mu(\si))))} \E{N(k)\mu^{N(k)}(\si)}. \label{eq:NkPlusTwo}
\end{align}
In the course of the previous calculations, we conditioned on the event $D(k)=t$ in order to incorporate the dependence between the downtimes. In the expression for $\E{p^{N(k+2)}|D(k) = t}$, we see that the value $t$ is only found in the form $e^{-st} (s\ge 0)$, meaning that unconditioning leads to expressions in terms of the LST $\tildeD(\cdot)$:
\begin{align}
\E{p^{N(k+2)}} &= \int_{t=0}^\infty \E{p^{N(k+2)}|D(k) = t} d\Pro{D(k)<t} \notag \\
&= E(p) \E{p^{N(k)}} + F(p) \E{\mu^{N(k)}(\si)} + G(p)\E{N(k)\mu^{N(k)}(\si)}, \label{eq:NkplusTwoInNk}
\end{align}
with 
\begin{align}
  {E(p)} =& \chi(\la(1-p))A^2(p)\tildeD(\la(1-p)+g(\la(1-p))), \label{eq:Ep}\\
  {F(p)} =& \chi(\la(1-p))K(p)\Big(A(p)\tildeD(\la(1-p)+g(\la(1-p))) \notag \\
  				& \qquad \qquad \qquad \qquad \qquad+\tildeD(\la(1-\mu(\si))+g(\la(1-p)))\lim_{z\rightarrow\mu(\si)}[A(z)+K(z)]\Big), \notag \\
  {G(p)} =& \chi(\la(1-p))K(p)\tildeD(\la(1-\mu(\si))+g(\la(1-p)))\notag \\
  				& \qquad \qquad \qquad \qquad \qquad *\frac{\si(1-\mu(\si))}{(\si+\la(1-\mu(\si)))(1+\lambda\tildeB(\si+\la(1-\mu(\si))))}. \notag
\end{align}
This expression gives a relation between $\E{p^{N(k+2)}}$ and $\E{p^{N(k)}}$.

\subsection{The queue length distribution at the beginning of an arbitrary uptime}\label{finalisingEPN}
We now compute $\E{p^N} = \lim_{k\rightarrow\infty}\E{p^{N(k)}}$. Combining \eqref{eq:id} and \eqref{eq:NkplusTwoInNk}, we find
\begin{equation}\label{eq:EPNPrelim}
\E{p^N} =\frac{F(p)\E{\mu^N(\si)}+G(p)\E{N\mu^N(\si)}}{1-E(p)}
\end{equation}
with $E(p), F(p)$ and $G(p)$ as before. Observe that this expression has two unknown constants $\E{\mu^{N}(\si)}$ and $\E{N\mu^{N}(\si)}$. We show that these constants can be obtained as the solution of a system of two linear equations. These two equations lead to a unique solution for $\E{\mu^{N}(\si)}$ and $\E{N\mu^{N}(\si)}$. We derive them below. Expressions for the constants immediately follow.

\paragraph*{The case $p=1$.} Since the left-hand side of \eqref{eq:EPNPrelim} evaluates to one for $p=1$ and $F(1) = G(1) = 1-E(1) = 0$, we have for the right-hand side
\begin{equation*}
\lim_{p\rightarrow 1} \big[\frac{F(p)\E{\mu^N(\si)}+G(p)\E{N\mu^N(\si)}}{1-E(p)}\big] = -\frac{F'(1)\E{\mu^N(\si)}+G'(1)\E{N\mu^N(\si)}}{E'(1)} = 1
\end{equation*}
by l'H\^opital's rule. Since $E(p)$, $F(p)$ and $G(p)$ are each differentiable at $p=1$, this results in the first linear equation in the two unknowns $\E{\mu^{N}(\si)}$ and $\E{N\mu^{N}(\si)}$.

\paragraph*{The case $p=\phi$.} The denominator $1-E(p)$ of \eqref{eq:EPNPrelim} has a root $p=\phi$ between zero and $\mu(\si) < 1$. More specifically, we have the following:
\begin{lemma}\label{lem:uniqueRootPhi}
The denominator $1-E(p)$ has exactly one root on the real line in the domain $(0,\mu(\si))$.
\end{lemma}
\begin{proof}
See Appendix \ref{proofLemmaDenominator}.
\end{proof}
Let $\phi$ be the unique root mentioned in Lemma \ref{lem:uniqueRootPhi}. Since $\E{p^N}$ is analytic in $p$ for $|p|\le 1$ and thus cannot evaluate to $\pm \infty$ for $0 < p < \mu(\si)$, we have that this root should also be a root for the numerator. Hence, we have that $F(\phi)\E{\mu^N(\si)}+G(\phi)\E{N\mu^N(\si)} = 0$.
 
Combining \eqref{eq:EPNPrelim} with the cases $p=1$ and $p=\phi$, we obtain the following lemma:

\begin{lemma}\label{lem:EPN}
The PGF of the queue length at the beginning of an arbitrary uptime is given by
\begin{equation*}
\E{p^N} =\frac{F(p)\E{\mu^N(\si)}+G(p)\E{N\mu^N(\si)}}{1-E(p)},
\end{equation*}
where
\begin{equation*}
\E{\mu^N(\si)} = \frac{E'(1)G(\phi)}{F(\phi)G'(1)-F'(1)G(\phi)} \mbox{  and   } \E{N\mu^N(\si)} = \frac{E'(1)F(\phi)}{F'(1)G(\phi)-F(\phi)G'(1)}.
\end{equation*}
\end{lemma}

\subsection{The queue length distribution at an arbitrary point in time}\label{obtainingEPL}
The main goal of this section is to determine the (PGF of the) queue length distribution at an arbitrary point in time. To do so, we expand the results of the previous section. The PGF $\E{p^M}$ of the queue length at the start of an arbitrary downtime is easily derived from the PGF $\E{p^N}$ of the queue length at the start of an arbitrary uptime. We then obtain the PGFs of the queue length when observed at an arbitrary point within an uptime and when observed at an arbitrary point within a downtime respectively. As a result, we finally obtain a general expression for $\E{p^L}$, the PGF of the queue length at an arbitrary point in time.

\subsubsection{Observing the queue length during an arbitrary uptime}

For the distribution of the queue length at an arbitrary point during an arbitrary uptime, we first obtain an expression for $\E{p^M}$. Then, we show that the PGF of the desired distribution equals this expression. 

Following the same reasoning as in Sections \ref{periodDK} and \ref{periodDKplusOne}, we derive $\E{p^M}$ by noting that during a downtime $D$ between $M$ and $N$, new customers arrive, but no customers are being processed:
\begin{align}
\E{p^{N}}    	 =& \int_{t=0}^\infty \E{p^{M}}\E{p^{H(t)}} d\Pro{D<t} \notag \\
							 =& \E{p^{M}}\tildeD(\la(1-p)) \label{eq:NinM},
\end{align}
where $H(t)$ is the number of Poisson ($\la$) arrivals during a time interval $t$. Thus, the following lemma can be derived:

\begin{lemma}\label{QLServerUp}
The PGF of the queue length at an arbitrary point in an uptime is given by
\begin{equation*}
\E{p^L|\textit{server up}} = \frac{\E{p^N}}{\tildeD(\la(1-p))},
\end{equation*}
where $\E{p^N}$ is given in Lemma \ref{lem:EPN} and $\tildeD(\cdot)$ satisfies \eqref{eq:identityD} and \eqref{eq:infprodform}.
\end{lemma}
\begin{proof}
Let $V(t)$ be the number of vacation initiations of the server in $(0,t]$. Note that $V(t)$ is a doubly stochastic process, where during a server uptime initiations of vacations occur according to a Poisson process with rate $\sigma$, whereas they obviously occur with rate zero when the server is already on a vacation. The conditional PASTA property (cf. \cite{conditionalPasta}) applied to $V(t)$ implies that the queue length distribution at the start of vacations equals the queue length distribution at an arbitrary point in time during an uptime. Hence, $\E{p^L|\textit{server up}} = \E{p^{M}}$. Combining this with \eqref{eq:NinM} yields the result.
\end{proof}

\begin{remark}
An expression for $\E{p^M}$ into $\E{p^N}$ is also readily given by \eqref{eq:periodUK}. This leads to an alternative expression for the PGF of the queue length when observed during a downtime:
\begin{equation*}
\E{p^L|\textit{server up}} = A(p)\E{p^N}+K(p)\E{\mu^N(\si)},
\end{equation*}
with $A(p)$, $K(p)$ and $\mu(\si)$ as before.
\end{remark}

\subsubsection{Observing the queue length during a downtime}
At an arbitrary point in time during a downtime, the number of customers in the system can be decomposed into the number of customers who were already waiting at the end of the previous uptime $M$, and the number of customers who arrived during the elapsed time $D^{past}$ since the start of the \emph{current} downtime, which we denote with $H(D^{past})$. Note that $M$ and $H(D^{past})$ are not independent. A large value of $M$ may imply that the \emph{previous} downtime has been very long. 
Due to the positive correlation between the downtimes as assumed in both models, this would in its turn imply that the \emph{current} downtime is probably longer than usual as well. The \emph{current} downtime and its past time $D^{past}$ are obviously dependent, which results in the fact that $M$ and $H(D^{past})$ are dependent. Using the notation illustrated in Figure \ref{fig:queueLength}, we obtain
\begin{align}
\E{p^L|\textit{server down}} &= \E{p^{M+H(D^{past})}}  \notag \\
&= \lim_{k\rightarrow\infty} \int_{0}^\infty \E{p^{M(k+1)}|D(k)=t}\E{p^{H(D^{past}(k+1))}|D(k) = t} d\Pro{D(k)<t}. \label{eq:queueLengthServerDown}
\end{align}
From the intermediate calculations leading to \eqref{eq:NkplusTwoInNk} (or by simply combining \eqref{eq:periodDKplusOne} and \eqref{eq:NkPlusTwo}), we have that
\begin{equation}\label{eq:form}
\lim_{k\rightarrow \infty} \E{p^{M(k+1)}|D(k) = t} = \sum_{i=1}^2 q_i(p)e^{-r_i(p)t},
\end{equation}
where
\begin{align}
q_1(p) =& A(p)(A(p)\E{p^N}+K(p)\E{\mu^N(\si)}), \\
q_2(p) =& K(p) \Big(\Big(\lim_{z\rightarrow\mu(\si)}[A(z)+K(z)]\Big)\E{\mu^N(\si)} \\
&\qquad\qquad\qquad\qquad+ \frac{\si (1-\mu(\si))}{(\si+\la(1-\mu(\si)))(1+\la \tildeB'(\si+\la(1-\mu(\si))))}\E{N\mu^{N}(\si)}\Big), \\
r_1(p) =& \la(1-p) \textit{ and } r_2(p)=\la(1-\mu(\si)).
\end{align}
Moreover, from \eqref{eq:combedecomposition} we obtain
\begin{align}
\lim_{k\rightarrow\infty} \E{p^{H(D^{past}(k+1))}|D(k) = t} &=\E{e^{-\la(1-p)D^{past}(k+1)}|D(k) = t} \notag \\ 
&= \frac{1-\E{e^{-\la(1-p)D(k+1)}|D(k) = t}}{\la(1-p)\E{D(k+1)|D(k)=t}} \notag \\
&= \frac{1-\chi(\la(1-p))e^{-g(\la(1-p))t}}{\la(1-p)(g'(0)t - \chi'(0))}. \label{eq:PGFHDPast}
\end{align}
Combining \eqref{eq:form}--\eqref{eq:PGFHDPast}, we have that the evaluation of \eqref{eq:queueLengthServerDown} involves the computation of a linear combination of integrals with the form 
\begin{equation*}
\int_{t=0}^\infty\frac{e^{-at}}{bt+c} d\Pro{D<t} = \int_{t=0}^\infty \int_{u=0}^\infty e^{-(at+(bt+c)u)} du \ d\Pro{D<t}.
\end{equation*}
By interchanging the integrals, this expression reduces to 
\begin{equation*}
\int_0^\infty e^{-cu}\tildeD(a+bu) du =: \kappa_{[a, b]}\{c\},
\end{equation*}
i.e. the Laplace transform of the function $\gamma_{[a,b]}\{u\} := \tildeD(a+bu)$. We obtain the following lemma:

\begin{lemma}\label{QLServerDown}
The PGF of the queue length at an arbitrary point in a downtime is given by
\begin{align}
\E{p^L|\textit{server down}} =& \int_{t=0}^\infty \sum_{i=1}^2 q_i(p)e^{-r_i(p)t}\frac{1-\chi(\la(1-p))e^{-g(\la(1-p))t}}{\la(1-p)(g'(0)t - \chi'(0))} d\Pro{D<t}    \notag\\
=& \frac{1}{\lambda(1-p)} \sum_{i=1}^2 q_i(p)\Big(\kappa_{[r_i(p), g'(0)]}\{-\chi'(0)\}\notag \\
&\qquad\qquad \qquad\qquad -\chi(\la(1-p))\kappa_{[r_i(p)+g(\la(1-p)), g'(0)]}\{-\chi'(0)\}\Big),
\end{align}
where $\gamma_{[a,b]}\{u\} = \tildeD(a+bu)$ and $\kappa_{[a, b]}\{c\} = \int_0^\infty e^{-cu}\tildeD(a+bu) du$, the Laplace transform of $\gamma_{[a,b]}\{\cdot\}$.
\end{lemma}
Note that in case $\tildeD(\cdot)$ is not explicitly known by inspecting \eqref{eq:identityD}, one can still evaluate $\kappa_{[a, b]}\{c\}$ up to arbitrary precision by truncating the infinite product form in \eqref{eq:infprodform}.

\subsubsection{Deriving the general queue length distribution}
Now that we have the PGF of $L$ conditioned on the state of the server, we readily have the queue length distribution of the single server model:

\begin{theorem}\label{thm:EPL}
For the PGF of the queue length in the single server model with one-dependent downtimes, we have
\begin{equation}\label{eq:EPL}
\E{p^L} = p_{up} \E{p^L|\textit{server up}} + p_{down} \E{p^L|\textit{server down}},
\end{equation} 
where
\begin{equation*}
p_{up} = \frac{\E{U}}{\E{U}+\E{D}} = \frac{1}{1+\si\E{D}} \textit{ and } p_{down} = \frac{\E{D}}{\E{U}+\E{D}} = \frac{\si\E{D}}{1+\si\E{D}}.
\end{equation*}
\end{theorem}
Expressions for $\E{p^L|\textit{server up}}$ and $\E{p^L|\textit{server down}}$ are given in Lemmas \ref{QLServerUp} and \ref{QLServerDown} respectively.  The weights $p_{up}$ and $p_{down}$ are the probabilities that one finds the server up and down respectively when observing the system at a random point in time in steady-state.
These probabilities are derived through the straightforward application of Palm theory (cf. \cite{BaccelliBremaud, Serfozo}) and involve the computation of $\E{U}$ and $\E{D}$. The former is determined by the fact that $U$ is exponentially ($\sigma$) distributed, and the latter follows from \eqref{eq:momentsD}.

The obtained expression for the (PGF of the) queue length distribution is exact for the single server vacation model with server vacation times dependent according to \eqref{eq:combedecomposition}. It can serve as an approximation for the marginal queue length distribution of a first-layer queue in the layered model, which is examined in Section \ref{approx}. We end this section with two remarks. 

\begin{remark}\label{rem:infprodform} Observe that the evaluation of \eqref{eq:EPL} involves the evaluation of several values of the downtime LST $\tildeD(\cdot)$. Whenever the downtime LST is not readily derived by \eqref{eq:identityD}, computing the values of $\tildeD(\cdot)$ is not possible in an exact fashion. However, we can use the infinite product form representation \eqref{eq:infprodform} to derive these values up to arbitrary precision. This product converges geometrically fast (as is common for such recursions as \eqref{eq:identityD}, which often arise in vacation-type models; see also \cite{BoxmaCombe}) and therefore truncation leads to an arbitrarily accurate approximation. The numerical experiments in Section \ref{numStudy} also confirm this fast convergence.
\end{remark}

\begin{remark} The analysis of the single server queue as presented in this section can be extended to other dependence forms than \eqref{eq:combedecomposition}. For  example, for Markov-modulated dependencies the same strategy can be used to obtain queue length distributions. Slight adaptations have to be made in the computations, starting with the conditional LST term in \eqref{eq:periodDKplusOne}. 
\end{remark}

\subsection{A note on the impact of dependence}\label{effectDependence}

Now that we have obtained the PGF of the queue length, we numerically study the influence of the downtime dependence on the queue length distribution. We will show that the level of dependence between the downtimes influences the queue length distribution considerably. Observe an instance of the single server model where $\lambda$ = 3, the service time $B$ is exponentially distributed with rate 5, and the uptime $U$ of the server  is exponentially distributed with rate 1/3. In this particular example, the downtime of the server consists of multiple exponential phases. The number of phases of which a downtime $D(k+1)$ consists, depends on the previous downtime $D(k)$:
\begin{equation}\label{eq:exampledowntimes}
D(k+1) \deq C_1 + \dots + C_{J(D(k))+1},
\end{equation}
where the $C_i$, which represent the phases, are i.i.d. exponentially ($\delta$) distributed, $\delta > 1$, and $J(D(k))$ is Poisson di\-stri\-bu\-ted with parameter $D(k)$. This implies that 
\begin{align*}
\E{e^{-sD(k+1)}|D(k) = t} &= \sum_{j=0}^\infty \E{e^{-s(C_1 + \sum_{i=2}^{j+1}C_i)}} e^{-t} \frac{t^j}{j!} \\
&= \E{e^{-sC_1}} \sum_{j=0}^\infty \E{e^{-sC_1}}^j e^{-t} \frac{t^j}{j!} \\
&= \E{e^{-sC_1}}e^{-(1-\E{e^{-sC_1}})t}.
\end{align*}
Therefore, we have that $\chi(s) = \E{e^{-sC_1}} = \frac{\delta}{\delta+s}$ and $g(s) = 1-\E{e^{-sC_1}} = \frac{s}{\delta+s}$. 
The stationary downtime is exponentially ($\de-1$) distributed, since \eqref{eq:identityD} is satisfied for its LST $\tildeD(s) = \frac{\delta-1}{\delta-1+s}$.  Observe that the stationary downtime distribution only exists for $\de>1$.

We compare the model above with its `independent counterpart', namely a single server queue with the same interarrival, service, uptime and stationary downtime distributions as before, but with mutually independent downtimes. The independent downtimes also fit in the dependence structure of \eqref{eq:combedecomposition} by simply setting $g(s) = 0$ for all $s$. Since the stationary downtime distribution is exponentially $(\de-1)$ distributed, we trivially have for the independent model that  $\chi(s) = \tildeD(s) = \frac{\de-1}{\de-1+s}$, $g(s) = 0$.

To see the effect of the dependencies, we compare the expected queue length in the dependent model, $\E{L^{dep}}$, with that of the independent model, $\E{L^{indep}}$. These values are obtained by evaluating the derivative of (\ref{eq:EPL}) at $p=1$. We compute the percentual relative difference of both quantities, i.e.
\begin{equation*}
\Delta := 100\% \times \frac{\E{L^{dep}}-\E{L^{indep}}}{\E{L^{indep}}},   
\end{equation*}
for varying values of $\delta$ such that the load of the system varies between 0.6 and 1. For the dependent model, the value of $\delta$ determines the correlation coefficient between two consecutive downtimes in steady-state, which we denote by $r$. More specifically, we have by the definition of the correlation coefficient that
\begin{equation}
r = \frac{\lim_{k \rightarrow \infty} \E{D(k)D(k+1)} - (\E{D})^2}{\E{D^2}-(\E{D})^2}.
\end{equation}
This expression can be further expressed in terms of $\delta$ by using \eqref{eq:momentsD} and \eqref{eq:jointDExpectation}. For the independent model, the correlation coefficient between the downtimes obviously equals zero at all times.
\begin{figure}
\begin{center}    
\includegraphics[width=\figurewidth]{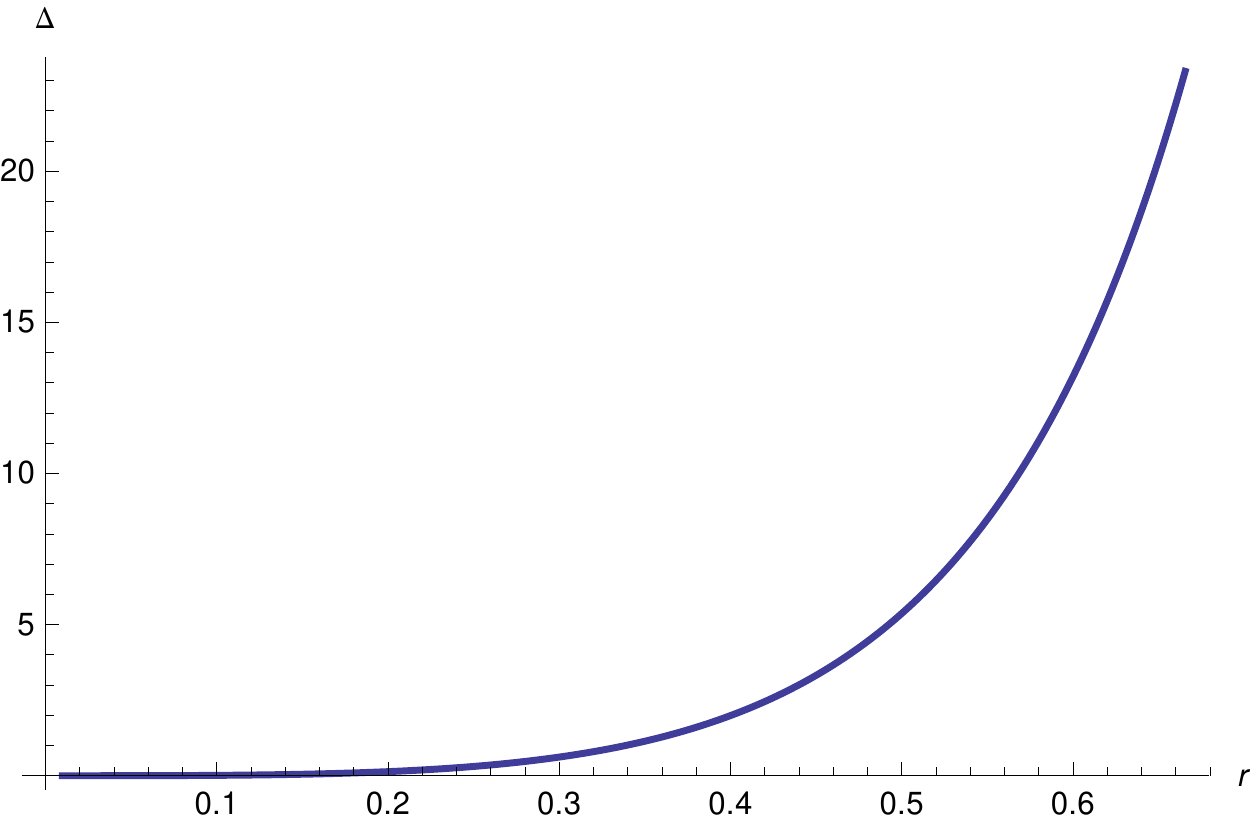}
\caption{The relative difference $\Delta$ in $\E{L}$ between the dependent and the independent model for various values of the correlation coefficient $r$.}
\label{fig:impactDowntime}
\end{center}
\end{figure}
Figure \ref{fig:impactDowntime} shows the value of $\Delta$ as a function of the correlation $r$ as observed in the dependent model. We see in this figure that $\Delta$ equals zero for $r=0$, while $\Delta$ grows as high as 25\% for increasing $r$. Thus, if we would approximate the mean queue length in the dependent model by ignoring the correlation between the downtimes, we can have an error of 25\%. This figure shows that the correlation in the downtimes can have a large impact on the queue length.

\section{Approximating queue lengths in the two-layered model} \label{approx}
In this section, we use Theorem \ref{thm:EPL} to obtain approximations for the marginal queue length distributions in the layered model. Arrival streams, service times and uptimes are equivalent for both models. To describe the downtime distribution and the dependence of the downtimes in the layered model in terms of the parameters of the single server model as well as possible, we need to obtain suitable choices for the functions $\chi(s)$ and $g(s)$, which are used in \eqref{eq:combedecomposition}. When investigating $L_i$, the queue length of $Q_i$, we choose functions $\chi_i(s)$ and $g_i(s)$ that are specific to $M_i$, $i=1,2$. The resulting explicit downtime structure matches the downtime distribution and downtime dependence of the downtimes of $M_i$ in the layered model closely, but does not model it exactly. Thus plugging in these suitable choices for $\chi_i(s)$ and $g_i(s)$ into \eqref{eq:EPL} only yields an \emph{approximation} for the distribution of $L_i$. However, numerical results in Section \ref{numStudy} will show this approximation to be very accurate.

The accuracy of the approximation suggested by Theorem \ref{thm:EPL} depends on the quality of the choices for $\chi_i(s)$ and $g_i(s)$. Therefore, we first focus on how to choose these functions appropriately. For this purpose, we compute in Section \ref{computationCorr} the first two moments and the correlation coefficient of subsequent downtimes in the layered model.  Based on these numbers, we derive suitable choices for the functions $\chi_i(s)$ and $g_i(s)$ in Section \ref{choosefunctions}, such that they match the situation in the layered model as well as possible. After these preliminary steps, we combine these results with those of the previous section to obtain an approximation for (the PGF of) the distribution of $L_i$, one of the main results of this paper, in Section \ref{resultingApproximation}. This approximation is applicable for the layered model with two machines and one single repairman. However, the approach we follow remains valid for more general models. We discuss this in Section \ref{modelExtensions}.

\subsection{Moments and the correlation coefficient of downtimes in the layered model}\label{computationCorr}
In this section, we focus on exponential repair times. The analysis can be extended to phase-type repair times, but at the cost of more cumbersome expressions that offer little additional insight. We derive the first two moments of the stationary downtime distribution of machine $M_1$ in the layered model, as well as the correlation coefficient between two subsequent downtimes $D_1(k)$ and $D_1(k+1)$ in steady-state (i.e., for $k \rightarrow \infty$). We do this by studying the bivariate LST $\E{e^{-sD_1(k)-zD_1(k+1)}}$. Evidently, a downtime $D_1(k)$ can be decomposed into a waiting time $W_1(k)$ and a repair time $R_1(k)$. The waiting time $W_1(k)$ is either zero when $M_2$ is operational at the time of breakdown of $M_1$, or amounts to an exponentially ($\nu_2$) distributed residual of the repair time of $M_2$ otherwise.

Assume that the repairman repairs $M_1$ and $M_2$ at rate $\nu_1$ and $\nu_2$ respectively. As noted before, machines interfere with each other in the layered model through their downtimes. More specifically, we have that a lengthy repair time of $M_1$ may increase the waiting time in the next downtime of $M_2$. At the same time, a lengthy downtime for $M_2$ (which might be due to a long repair time) may have an increasing influence on the next waiting time of $M_1$. Therefore, we have that $R_1(k)$ and $W_1(k+1)$ are positively correlated. Keeping this in mind, the bivariate LST of two consecutive downtimes is
\begin{align}
\E{e^{-sD_1(k)-zD_1(k+1)}} =& \E{e^{-sW_1(k)}}\E{e^{-zR_1(k+1)}} \int_{0}^\infty e^{-sy}\E{e^{-zW_1(k+1)}|R_1(k) = y} \nu_1 e^{-\nu_1y} \,dy.  \label{eq:jointLST}
\end{align}
Since $R_1(k+1)$ is exponentially ($\nu_1$) distributed, only the terms $\E{e^{-sW_1(k)}}$ and $\E{e^{-zW_1(k+1)}|R_1(k) = y}$ remain to be computed.

First, we derive $\E{e^{-sW_1(k)}}$. Just before $M_1$ breaks down, either $M_2$ is up and running, or $M_2$ is in repair. The probability of either event happening is derived by studying the embedded Discrete Time Markov Chain (DTMC) of the machine states at epochs where any machine breaks down, starts being repaired or ends a repair period. 
Let $\vec{X}_n = \{X_{1,n}, X_{2,n}\}$ denote the state of the machines after the $n$-th transition. We represent the state of $M_i$ being up at time $t$, waiting for repair or being in repair, by $X_{i,n} = 1$, $X_{i,n} = 2$ or $X_{i,n} = 3$ respectively. Since life times and repair times of both machines are exponential, we have that $\{\vec{X}_n, n\ge0\}$ is a DTMC on the state space $S=\{\{1, 1\}, \{1, 3\}, \{3, 1\}, \{2, 3\}, \{3, 2\}\}$. It naturally follows that the non-zero transition probabilities $p_{i, j}$ from state $i$ to state $j$ are given by $p_{\{1, 1\}, \{3, 1\}} = 1-p_{\{1, 1\}, \{1, 3\}} = \frac{\sigma_1}{\sigma_1+\sigma_2}$, $p_{\{1, 3\}, \{1, 1\}} = 1-p_{\{1, 3\}, \{2, 3\}} = \frac{\nu_2}{\sigma_1+\nu_2}$, $p_{\{3, 1\}, \{1, 1\}} = 1-p_{\{3, 1\}, \{3, 2\}} = \frac{\nu_1}{\nu_1+\sigma_2}$ and $p_{\{2, 3\}, \{3, 1\}} = p_{\{3, 2\}, \{1, 3\}} = 1$. The DTMC is irreducible and aperiodic, hence a unique limiting distribution $\pi$ on $S$ exists and can be derived. Given this distribution function, the probability of an arbitrary transition being an event where $M_1$ breaks down equals $\pi_{\{1, 1\}}p_{\{1, 1\}, \{3, 1\}} + \pi_{\{1, 3\}}p_{\{1, 3\}, \{2, 3\}}$. The probability $z_{up}$ ($z_{down}$) of $M_2$ working (being in repair), given that $M_1$ breaks down next transition, is thus given by
\begin{align*}
z_{up} &= \frac{\pi_{\{1, 1\}}p_{\{1, 1\}, \{3, 1\}}}{\pi_{\{1, 1\}}p_{\{1, 1\}, \{3, 1\}} + \pi_{\{1, 3\}}p_{\{1, 3\}, \{2, 3\}}} = \frac{\si_1\nu_1+(\si_2+\nu_1)\nu_2}{\left(\sigma _2+\nu _1\right)\left(\sigma _1+\sigma _2+\nu _2\right)},\\
z_{down} &= \frac{\pi_{\{1, 3\}}p_{\{1, 3\}, \{2, 3\}}}{\pi_{\{1, 1\}}p_{\{1, 1\}, \{3, 1\}} + \pi_{\{1, 3\}}p_{\{1, 3\}, \{2, 3\}}} = \frac{\sigma _2\left(\sigma _1+\sigma _2+\nu _1\right)}{\left(\sigma _2+\nu _1\right)\left(\sigma _1+\sigma _2+\nu _2\right)}.
\end{align*}
Hence $M_1$ has to wait with probability $z_{down}$, whereas it does not with probability $z_{up}$. Therefore, we have that
\begin{equation*}
\E{e^{-sW_1(k)}} = z_{up} + z_{down} \frac{\nu_2}{\nu_2+s} = \frac{s \si_1 \nu_1+s (\si_2+\nu_1) \nu_2+(\si_2+\nu_1) \nu_2 (\si_1+\si_2+\nu_2)}{(\si_2+\nu_1) (s+\nu_2) (\si_1+\si_2+\nu_2)}.
\end{equation*}
For $\E{e^{-zW_1(k+1)}|R_1(k) = y}$, we first conclude that at the moment $M_1$ is taken into repair for $y$ time units, $M_2$ must be working. After these $y$ time units, we have a probability $e^{-\nu_2y}$ of $M_2$ having broken down in the mean time, whereas it is still functioning with probability $1-e^{-\nu_2 y}$. Given the former event that $M_2$ is still working at the end of $R_1(k)$, there is a probability $v$ that $M_2$ is in repair when $M_1$ breaks down again, i.e. at the start of $W_1(k+1)$. Due to the memoryless property of the exponential distribution, this probability $v$ is easily determined by the fixed point equation
\begin{equation*}
v = \frac{\si_2}{\si_1+\si_2} \left(\frac{\si_1}{\si_1+\nu_2} + \frac{\nu_2}{\si_1+\nu_2}v\right) \Rightarrow v = \frac{\si_2}{\si_1+\si_2+\nu_2}.
\end{equation*}
This allows us to determine the probability $w$ that $M_2$ is in repair at the start of $W_1(k+1)$, given that $M_2$ was waiting for repair at the end of $R_1(k)$:
\begin{equation*}
w = \frac{\si_1}{\si_1+\nu_2} + \frac{\nu_2}{\si_1+\nu_2}v = \frac{\si_1+\si_2}{\si_1+\si_2+\nu_2}.
\end{equation*}
Taking these probabilities together, we have that $W_1(k+1)$ is exponentially ($\nu_2$) distributed with probability $e^{-\nu_2y}v + (1-e^{-\nu_2 y})w$ and zero with probability $e^{-\nu_2y}(1-v) + (1-e^{-\nu_2 y})(1-w)$. Thus, 
\begin{align*}
\E{e^{-zW_1(k+1)}|R_1(k) = y} =& \left(e^{-\nu_2y}v + (1-e^{-\nu_2 y})w\right)\frac{\nu_2}{\nu_2+z} \\
&+ e^{-\nu_2y}(1-v) + (1-e^{-\nu_2 y})(1-w) \\
=& e^{-\nu_2y} \frac{\si_2\nu_2+(\si_1+\nu_2)(\nu_2+z)}{(\si_1+\si_2+\nu_2)(\nu_2+z)} \\
&+ (1-e^{-\nu_2y})\frac{(\si_1+\si_2)\nu_2+(\nu_2+z)\nu_2}{(\si_1+\si_2+\nu_2)(\nu_2+z)}.
\end{align*}
One can now compute $\E{e^{-(sD_1(k)+zD_1(k+1))}}$ using \eqref{eq:jointLST}. 
By differentation, we obtain the moments of $D_1$ and the autocovariance
\begin{equation}\label{eq:covLayeredModel}
\Cov{D_1(k)}{D_1(k+1)} = \frac{\si_1\si_2}{(\si_2+\nu_1)^2\nu_2(\si_1+\si_2+\nu_2)}.
\end{equation}
The correlation coefficient between $D_1(k)$ and $D_1(k+1)$ is now obtained by dividing this expression by the variance of the stationary downtime $D$. Now that the first two moments of the stationary downtime distribution as well as the correlation coefficient are known, we are in a position to approximate the length of $Q_1$ in the layered model with the result on the queue length in the single server model. 

\begin{remark}
The covariance as given in \eqref{eq:covLayeredModel} and the resulting correlation coefficient both evaluate to zero when $\si_1$ or $\si_2$ is zero, or when $\si_2$, $\nu_1$ or $\nu_2$ tends to infinity. If either $\si_1$ or $\si_2$ is zero, one of the machines essentially never breaks down and there is no interference between the machines. When $\si_2$ tends to infinity, there is no correlation in the downtimes of $M_1$ either, since $M_2$ is practically always down. Therefore, every single downtime of $M_1$ will consist of a repair time of $M_1$ plus a residual repair time of $M_2$, which are both independent of anything else. When $\nu_1$ tends to infinity, $M_1$ essentially does not require any repair time from the repairman and $M_2$ will never have to wait for the repairman to become idle. As a result, the downtimes of $M_2$ are independent. A waiting time for $M_1$ then comes down to either zero when $M_2$ is up, or the residual of an $M_2$ repair, of which the starting time is not biased by the breakdowns of $M_1$. This implies that downtimes of $M_1$ are independent as well in that case. Equivalently, when $\nu_2$ tends to infinity, $M_2$ does not require any repair time from the repairman, which means that downtimes of $M_2$ do not influence downtimes of $M_1$. As a result, there is no correlation in the $M_1$ downtimes in this case either. Furthermore, both the covariance and the resulting correlation coefficient are increasing in $\si_1$ and decreasing in both $\nu_1$ and $\nu_2$, due to similar arguments as the above. 
\end{remark}

\begin{remark}
In case $\si_1 = \si_2$ and $\nu_1 = \nu_2$, the LST $\E{e^{-sW_1(k)}}$ can also be obtained using the arrival theorem (cf. \cite{LavenbergReiser}), which
states that in a closed queueing network, the stationary state probabilities at instants at which customers arrive at a service unit are equal to the stationary state probabilities at arbitrary times for the network with one less customer. This implies that at time epochs $M_1$ breaks down, the probability distribution on the state of $M_2$ (either up or in repair) is equal to the steady-state distribution of the state of $M_2$ in a system with $\si_1 = 0$, but with $\si_2$ and $\nu_2$ left unchanged. In such a system, $M_2$ is the only machine requiring attention of the repairman, which greatly simplifies the analysis.
\end{remark}

\subsection{Choosing the appropriate dependence functions}\label{choosefunctions}
In order to use Theorem \ref{thm:EPL} as an approximation for the PGF of $L_i$ in the layered model, we need to identify suitable expressions for the functions $\chi_i(s)$ and $g_i(s)$. These functions need to match the dependence in the downtimes of $M_i$ as well as possible, or equivalently, the expressions for \eqref{eq:jointDLST} and \eqref{eq:jointLST} need to coincide as well as possible. The quality of the choices for the functions directly influences the accuracy of the approximation, as they are the only source of error introduced. In order to obtain suitable expressions for $\chi_i(s)$ and $g_i(s)$, we perform two-moment fits commonly used in literature. To this end, the first two moments of the distributions represented by the LST's $\chi_i(s)$ and $e^{-g_i(s)}$ must be determined. We do this based on expressions for $\chi_i'(0)$, $\chi_i''(0)$, $g_i'(0)$ and $g_i''(0)$, which we obtain by combining \eqref{eq:momentsD} and \eqref{eq:jointDExpectation} with results for the first two moments of the downtime distribution and the correlation coefficient of the consecutive downtimes. These depend on the distributions of the repair times $R_1$ and $R_2$, among others. For exponential repair time distributions, the results required were obtained in Section \ref{computationCorr} by inspection of the embedded Markov chain. By using the same methods, similar results can be obtained for phase-type repair times.

\subsubsection{Obtaining derivatives of the dependence functions}
To obtain values for $\chi_i'(0)$, $\chi_i''(0)$, $g_i'(0)$ and $g_i''(0)$, we solve a set of equations. In Section \ref{computationCorr}, we have expressed $\E{D_i}$, $\E{D_i^2}$ and $\lim_{k \rightarrow \infty} \E{D_i(k)D_i(k+1)}$ in terms of the parameters of the layered model. By \eqref{eq:momentsD} and \eqref{eq:jointDExpectation}, we have that these expressions are related to the functions $\chi_i(\cdot)$ and $g_i(\cdot)$ as follows:
\begin{align}
\E{D_i} =& \frac{\chi_i'(0)}{g_i'(0)-1}, \notag \\
\E{D_i^2} =& \frac{\chi_i''(0)-\E{D_i}(2\chi_i'(0)g_i'(0)+g_i''(0))}{1-g_i'(0)^2}, \label{eq:fittingMoments} \\
\E{D_i(k)D_i(k+1)} =& -\chi_i'(0)\E{D_i} + g_i'(0)\E{D_i^2}. \notag
\end{align}
These three equations in four unknowns fix values for $\chi_i'(0)$ and $g_i'(0)$, but leave one degree of freedom in the determination of $\chi_i''(0)$ and $g_i''(0)$. This freedom can be used to fine-tune the model. For example, one might assume the independent component of the downtime to be distributed according to a certain distribution. This would lead to an additional equation for $\chi_i''(0)$ in terms of $\chi_i'(0)$, which then also fixes values for $\chi_i''(0)$ and $g_i''(0)$.

\subsubsection{Expressions for the dependence functions}

We now determine suitable expressions for $\chi_i(\cdot)$ and $g_i(\cdot)$. For this purpose, there are many approaches possible. Below, we base the choices of $\chi_i(\cdot)$ and $g_i(\cdot)$ on two-moment approximations. To apply these two-moment approximations, we use the squared coefficient of variation (SCV), which for a random variable $Z$ is equal to $\SCV{Z} = \Var{Z}/\E{Z}^2 = \frac{\E{Z^2}}{\E{Z}^2}-1$.

Observe that in the previous section, we already calculated the ingredients needed to obtain the first two moments of the distributions represented by the LST's $\chi_i(s)$ and $e^{-g_i(s)}$. As explained in Section \ref{notation}, the function $\chi_i(s)$ is the LST of a random variable representing the independent component of the downtime, with the first two moments given by $-\chi_i'(0)$ and $\chi_i''(0)$ respectively, and consequently with an SCV of $\frac{\chi_i''(0)}{(\chi_i'(0))^2}-1$. The function $e^{-g_i(s)}$ is the LST of an infinitely divisible distribution, the distribution of the incremental component of $D(k+1)$ per unit of $D(k)$, with as first two moments $g_i'(0)$ and $(g_i'(0))^2-g_i''(0)$ respectively, and therefore an SCV of $-\frac{g_i''(0)}{(g_i'(0))^2}$. 

Based on the two moments and the SCV for each of the distributions, we employ commonly used distributional two-moment fit approximations as described in \cite[p.\ 358--360]{Tijms}. For e.g.\ an SCV smaller than one, one fits a mixture of an Erlang($k, \ga$) and an Erlang($k-1, \ga$) distribution to the moments ($k\ge 2, \ga > 0$), whereas for an SCV larger than one, one uses a $H_2$ distribution with balanced means. In the special case of an SCV of zero or one, one uses a deterministic or exponential distribution respectively. The parameters for each of these distributions are based on the first two moments which are given as an input for this procedure.

Thus we choose the functions $\chi_i(s)$ and $g_i(s)$ as follows: first we compute the moments (Section \ref{computationCorr}), which we use in \eqref{eq:fittingMoments} to find the first two derivatives of $\chi_i(s)$ and $g_i(s)$. Based on these derivatives, we then fit repair-time distributions using the two-moment approximations in \cite[p.\ 358--360]{Tijms}. Recall that we assumed in Section \ref{notation} that $g_i(s)$ has a completely monotone derivative, so that the LST $e^{-g_i(s)}$ represents an infinitely divisible distribution. We show that the distributions derived by the two-moment approximation satisfy this assumption:
\begin{itemize}[topsep=0pt, itemsep=0pt, leftmargin=*]
\item For a deterministic distribution with value $x$ and LST $e^{-sx}$, we have $g_i(s) = sx$. This function obviously has a completely monotone derivative, since $\frac{d}{d s} g_i(s) = x \ge 0$ and $\frac{d^n }{d s^n} g_i(s) = 0$ for all $n \ge 2$.
\item For an exponential distribution and a $H_2$ distribution, see \cite[p. 452]{Feller} on mixtures of exponential distributions.
\item A mixture of an Erlang($k, \ga$) and an Erlang($k-1, \ga$) distribution with weights $q \in [0, 1]$ and $1-q$ respectively has LST $q\Big(\frac{\ga}{\ga+s}\Big)^k+(1-q)\Big(\frac{\ga}{\ga+s}\Big)^{k-1}$. Hence, $g_i(s) = -\log\Big(q\Big(\frac{\ga}{\ga+s}\Big)^k+(1-q)\Big(\frac{\ga}{\ga+s}\Big)^{k-1}\Big)$. Moreover we have that 
\begin{equation*}
\frac{d^n}{d s^n}g_i(s) = (-1)^{n+1} (n-1)!\Big(\frac{k}{(\ga+s)^n}-\frac{(1-q)^n}{(\ga+(1-q)s)^n}\Big).
\end{equation*}
The second term $(n-1)!$ is positive, as well as the third term, since $(\ga+s)^n \frac{(1-q)^n}{(\ga+(1-q)s)^n} \le \frac{(\ga+(1-q)s)^n}{(\ga+(1-q)s)^n} = 1 < 2 \le k$. Therefore, derivatives of odd order are positive through the first term, and negative otherwise. Hence $g_i(s)$ has a completely monotone derivative.
\end{itemize}

\begin{example}
If for the independent component we find that the SCV $\frac{\chi_i''(0)}{(\chi_i'(0))^2}-1$ equals one, then we fit an exponential distribution with rate $(-\chi_i'(0))^{-1}$. This distribution has LST $\frac{1}{1-s \chi_i'(0)}$, which is a suitable approximation for $\chi_i(s)$. Likewise, if for the incremental component we find that the SCV $-\frac{g_i''(0)}{(g_i'(0))^2}$ equals one, again an exponential distribution is fitted with rate $(g_i'(0))^{-1}$, which provides the suggestion $e^{-g_i(s)} = \frac{1}{1+s g_i'(0)}$, or equivalently $g_i(s) = \log[1+s g_i'(0)]$.
\end{example}

\subsection{Resulting approximation}\label{resultingApproximation}
Now that we have obtained expressions for $\chi_i(s)$ and $g_i(s)$, Theorem \ref{thm:EPL} directly forms an approximation for the (PGF of the) marginal queue length distribution in the layered model:
\begin{apprx}\label{apprx:approx}
In the two-layered model, an approximation $L_{i,app}$ for the queue length of $Q_i$ is given by the PGF
\begin{equation}\label{eq:EPLapp}
\E{p^{L_{i, app}}} = p_{up} \E{p^M} + p_{down} \E{p^{M+H(D^{past})}},
\end{equation}
where the expressions for $p_{up}, p_{down}, \E{p^M}$ and $\E{p^{M+H(D^{past})}}$ are as given in Section \ref{analysisSSQ}, but with $\la$, $\tildeB(\cdot)$, $\si$, $\chi(\cdot)$ and $g(\cdot)$ replaced by the two-layered model counterparts $\la_i$, $\tildeB_i(\cdot)$, $\si_i$, $\chi_i(\cdot)$ and $g_i(\cdot)$. 
\end{apprx}

\begin{remark}\label{rem:reasonApprox}
Note that \eqref{eq:jointLST} cannot be rewritten in the form of the bivariate LST \eqref{eq:jointDLST}, i.e., the dependence structure we assumed in \eqref{eq:combedecomposition} or \eqref{eq:jointDLST} does not model perfectly the distribution and the interdependence of the downtimes of $M_i$. In addition to this modelling approximation, a numerical approximation is introduced by truncation of the infinite product in \eqref{eq:infprodform}. However, the latter error can be made negligibly small.
\end{remark}

\subsection{Approximations for generalisations of the layered model}\label{modelExtensions}
Throughout the previous sections, we derived an approximation for the layered model with two machines {and a single repairman}. However, the approach followed can be readily extended to approximate queue lengths of first-layer queues in an equivalent model with a larger number of queues and machines {and multiple repairmen. Moreover, the approach followed in Section \ref{computationCorr} for deriving the moments and the correlation coefficient of the downtimes remains valid when assuming phase-type repair time distributions.} We discuss these model generalisations below. Note that in {the cases below}, we only apply the analysis on the single server queueing model as given in Section \ref{analysisSSQ} without any modification.

\paragraph*{Larger numbers of machines and first-layer queues.} When we generalise the layered model as described in Section \ref{notation} to allow for $N>2$ machines $M_1, \ldots, M_N$ and thus $N$ first-layer queues $Q_1, \ldots, Q_N$, we can still use Approximation \ref{apprx:approx} like before to approximate the PGFs of $L_1, \ldots, L_N$. The approach for deriving appropriate functions for $\chi_i(s)$ and $g_i(s)$, $i=1, \ldots, N$ needed to use Approximation \ref{apprx:approx} remains largely the same. 
However, by introducing a larger number of machines, the computation of the first two moments and the correlation coefficient of downtimes in the layered model becomes increasingly cumbersome. As opposed to the case $N=2$ as assumed in Section \ref{computationCorr}, the repair buffer can now contain multiple machines. Since the repair facility serves the queue in a First-Come-First-Serve (FCFS) manner, the order in which the machines are waiting for repair needs to be included in the state space of the embedded DTMC describing the states of the machines. Subsequently, considerably more conditioning is needed to compute the terms $\E{e^{-sW_1(k)}}$ and $\E{e^{-zW_1(k+1)}|R_1(k)=y}$ in \eqref{eq:jointLST} and ultimately the moments and the correlation coefficient of the downtimes.

\paragraph*{Multiple repairmen.} In the layered model, it is assumed there is only one repairman assigned to repair machines. This assumption can be relaxed to allow for $K>1$ repairmen in the repair facility, each working on a different machine and taking the broken machines out of the repair buffer in a FCFS manner. When $K \ge N$, a broken machine will always be taken into repair immediately. As a result, machines do not compete for repair facilities anymore, and consecutive downtimes of a machine become independent. Therefore, when taking $\chi_i(s)$ such that it equals the LST of the repair time distribution of $M_i$ and taking $g_i(0)=0$, the \textit{exact} PGF of $L_i$ is given by Approximation \ref{apprx:approx}. When $N>K$, consecutive downtimes of the machine remain correlated. Again, the approximation as developed in this paper remains valid, but difficulties arise in deriving the appropriate functions for $\chi_i(s)$ and $g_i(s)$, $i=1, \ldots, N$. More specifically, the computation of the moments and the correlation coefficient of the consecutive downtimes of each of the machines becomes again increasingly complicated. Since machines can now be repaired simultaneously, the order in which machines return to service after repair is not necessarily the same as the order in which machines break down. This introduces extra conditioning in e.g. the computation of $\E{e^{-zW_1(k+1)}|R_1(k)=y}$ in \eqref{eq:jointLST}, since the machines which were already waiting for repair at the start of $W_1(k)$ may not have returned to an operational state again by the time $R_1(k)$ has passed. This evidently influences $W_1(k+1)$.

\paragraph*{Phase-type distributed repair times.} In Section \ref{computationCorr}, we derived an explicit expression for the correlation coefficient of consecutive downtimes of a machine, in case repair times are exponentially distributed. For phase-type repair-time distributions, a similar approach for studying the embedded Markov chain can be followed to obtain the numbers needed to construct the functions $\chi_i(s)$ and $g_i(s)$ in Section \ref{choosefunctions}. The computations may become more involved, but remain conceptually the same. This leads to a more complicated expression for $\E{e^{-sW_1(k)}}$ in \eqref{eq:jointLST}. For the computation of $\E{e^{-zW_1(k+1)}|R_1(k)=y}$, extra conditioning on the repair phase is also needed.

\section{Numerical Study}\label{numStudy}
We now give some numerical examples to assess the accuracy of Approximation \ref{apprx:approx}. In Section \ref{initialGlance}, we compare our approximation for the marginal queue length to simulation results for a typical setting. Then, in Section \ref{testbed}, we observe the effect of the model parameters and identify several key factors determining the accuracy of the approximation.

\subsection{Initial glance at the approximation}\label{initialGlance}
Consider a system where $\la_1 = 0.25$, $\si_1 = \si_2 = 1$ and $B_1$, $R_1$ and $R_2$ are exponentially (1) distributed. Note that the settings for $\la_2$ and $B_2$ do not influence the length of $Q_1$. In Figure \ref{fig:singleInstance} we plot the approximated PGF of $L_{1, app}$ and the PGF of $L_1$ obtained by simulation. 
\begin{figure}
\begin{center}
\includegraphics[width=\figurewidth]{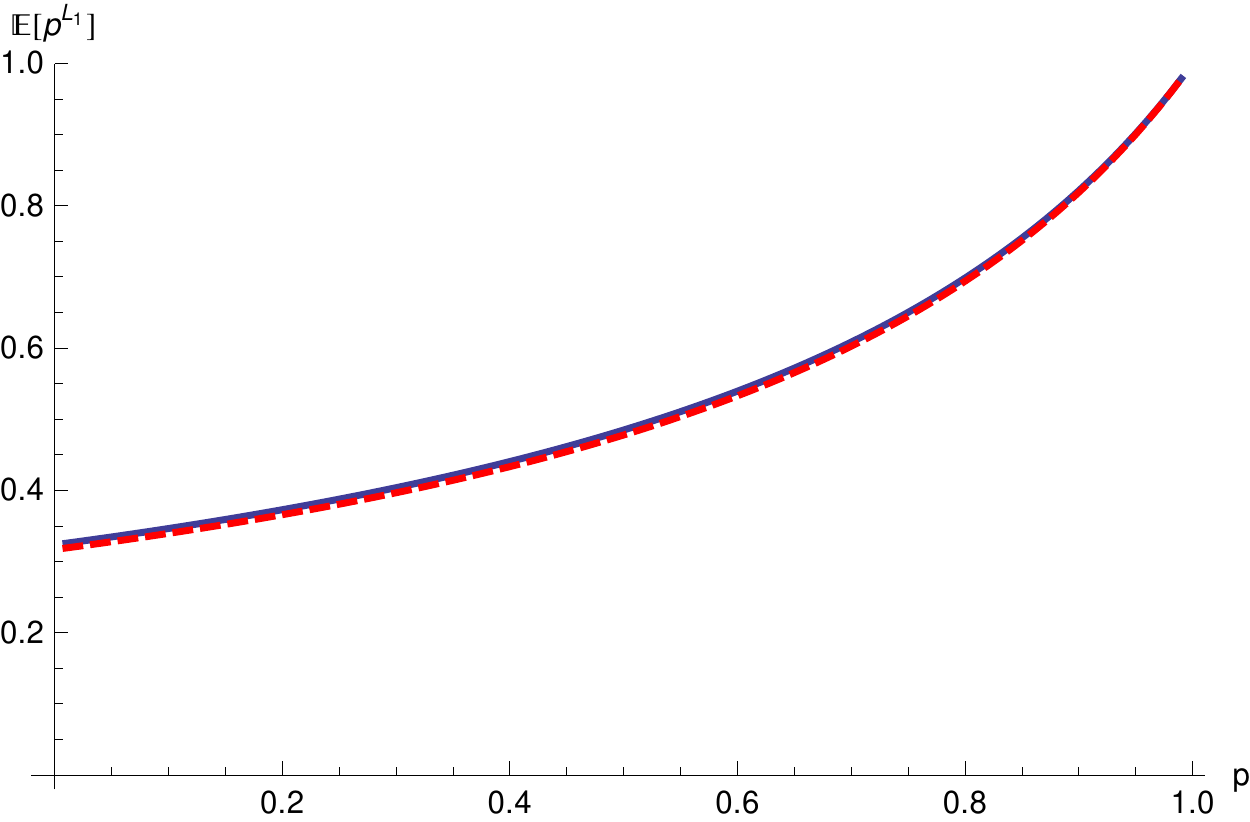}
\caption{Plot of $\E{p^{L_{1, app}}}$ (solid curve) and $\E{p^{L_1}}$ (dashed curve).}
\label{fig:singleInstance}
\end{center}
\end{figure}

We observe in this figure that $\E{p^{L_{1, app}}}$ matches $\E{p^{L_1}}$ very closely. The error made is largest at $p=0$, where $\E{p^{L_{1, app}}}$ is 2.09$\%$ larger than the value of $\E{p^{L_1}}$. The error decreases in $p$, which is why $\E{L}$ is approximated well. We have that $\E{L_{1, app}} = \frac{d}{dp} \E{p^{L_{1, app}}}|_{p=1} =$ 2.205, while the theoretical mean $\E{L}$ equals 2.220. This is a typical performance of the approximation. As we will see in the next section, the accuracy of the approximation can become worse if the downtimes in the layered model are extraordinarily correlated (cf. Figure \ref{fig:usVsWartenHorst}). Nonetheless, in realistic systems, even in the worst-case scenarios, the difference in the expected queue lengths is about 10\%. 

\subsection{Accuracy of the approximation}\label{testbed}
We now turn to the study of the parameter effects on the accuracy of the approximation. As we will see, the approximation performs very well over a wide range of parameter settings. Moreover, we observe several parameter effects. For example, the loads offered to the queues have little impact on the accuracy of the approximation, whereas the difference of time scales of product arrivals and services compared to that of machine breakdowns and repair, as well as the level of dependence between the downtimes, affect the accuracy significantly.

To study the accuracy of the approximation overall, we compare the approximated values for the mean of $L_1$ with the values obtained by numerical methods such as simulation or the power-series algorithm (see e.g. \cite{BlancOverview} or \cite{DorsmanVDMeiVlasiou}) in various instances of the two-layered model. In particular, we regard instances where $B_1$ is exponentially $(\delta_1)$ distributed, and $R_1$ and $R_2$ are exponentially distributed with rates $\nu_1$ and $\nu_2$ respectively. The complete test bed of instances that are analysed contains 675 different combinations of parameter values, all listed in Table \ref{table:paramValues}. This table lists multiple values for the workload of $Q_1 (\rho_1)$, the breakdown rates of $M_1$ and $M_2$ ($\si_1$ and $\si_2$) and the repair rates of $M_1$ and $M_2$ ($\nu_1$ and $\nu_2$). In particular, these rates are varied in the order of magnitude through the values $a^\si_i$ and $a^\nu_i$ as specified in the table and in the imbalance, through the values $b^\si_j$ and $b^\nu_j$. As a consequence, the breakdown rates $(\si_1, \si_2)$ and the repair rates $(\nu_1, \nu_2)$ run from $(0.1, 0.1)$, being small and perfectly balanced, to $(50, 10)$, being large and significantly imbalanced. Note that the values for $\rho_2$ and the service time distribution for $Q_2$ do not influence $L_1$, and hence are left unspecified.
\begin{table}[t]
\centering
\begin{tabular}{c|c}
\hline
Parameter & Considered parameter values \\ [0.5ex]
\hline
$\rho_1$ & $\{0.25, 0.5, 0.75\}$ \\
& \\
$\delta_1$ & $\{1\}$ \\
& \\
$(\si_1, \si_2)$ & $a^\si_i \cdot b^\si_j \qquad\forall i,j$ \\
								 & where $\vec{a}^\si = \{0.1, 1, 10\}$ and \\
								 & $\vec{b}^\si = \{(1, 1), (1, 2), (2,1), (1, 5), (5, 1)\}$ \\ \\
$(\nu_1, \nu_2)$ & $a^\nu_i \cdot b^\nu_j \qquad\forall i,j$ \\
								 & where $\vec{a}^\nu = \{0.1, 1, 10\}$ and \\
								 & $\vec{b}^\nu = \{(1, 1), (1, 2), (2,1), (1, 5), (5, 1)\}$ \\ \\
\hline
\end{tabular}
\caption{Parameter values of the test bed used to compare the approximation to exact results.}\label{table:paramValues}
\end{table}

For the systems corresponding to each of the parameter combinations in Table \ref{table:paramValues}, we compare the \emph{approximated} mean queue lengths of the first queue, $\E{L_{1, app}} = \frac{d}{dp} \E{p^{L_{1, app}}}|_{p=1}$ to the actual mean queue length $\E{L_1}$. Subsequently, we compute the relative error of these approximations, i.e. 
\begin{equation}\label{eq:definitionRelError}
\Delta := 100\% \times \Big|\frac{\E{L_{1, app}}-\E{L_1}}{\E{L_1}}\Big|. 
\end{equation}
In Table \ref{table:errorBreakdown} the resulting relative errors are summarised. We note that none of these errors is greater than 5$\%$, and the majority of these errors does not exceed 1$\%$. This seems to remain the case even as the load goes to one, or for extreme values of the imbalance in the system. One expects that for extremely small values of the system (i.e., for $a^\si_i, a^\nu_i \rightarrow 0$), the approximation becomes progressively worse. However, these values do not represent realistic systems. These results show that Approximation \ref{apprx:approx} works very well for typical systems.

\begin{table}
\centering
\begin{tabular}{|c||c|c|c|c|c|c|}
\hline
& 0-0.01\%	& 0.01-0.1\% & 0.1-1\% & 1-5\% & $>$5\%+ \\
\hline
\% of rel. errors & 25.93\% & 32.30\% & 32.15 \% & 9.63 \% & 0.00\% \\
\hline
\end{tabular}
\caption{Relative errors of the mean queue length approximation categorised in bins.}\label{table:errorBreakdown}
\end{table}

To observe any parameter effects, we also give the mean relative error categorised in some of the variables in Table \ref{table:categorisedInParameterValues}. From Table \ref{table:categorisedInParameterValues}\subref{table:categorisedInRho}, we see that the accuracy of the approximation is not very sensitive to the load of the queue. From Tables \ref{table:categorisedInParameterValues}\subref{table:categorisedInFSigma} and \ref{table:categorisedInParameterValues}\subref{table:categorisedInFNu} we however note that the orders of magnitude of the breakdown and repair rates do impact the accuracy of the approximation. This is due to the fact that the rate at which products move (i.e., arrive and get served) with respect to the life and repair times of the machine do differ in these cases. From Tables \ref{table:categorisedInParameterValues}\subref{table:categorisedInMSigma} and \ref{table:categorisedInParameterValues}\subref{table:categorisedInMNu}, we see that the imbalance of the breakdown and repair rates do impact the accuracy as well (but to a lesser extent). We discuss the observed effects in more detail below.

\begin{table}[ht!]
    \begin{center}
        \subtable[]
        {
            \begin{tabular}{|c||c|c|c|}
						
						\hline
						$\rho_1$ & 0.25 & 0.5 & 0.75 \\
						\hline
						Mean rel. error & 0.328\% & 0.316\% & 0.335\% \\
						\hline
						\end{tabular}
						\label{table:categorisedInRho}
        } \\
        \subtable[]
        {
            \begin{tabular}{|c||c|c|c|}
						
						\hline
						$a^\si_i$ & 0.1 & 1 & 10 \\
						\hline
						Mean rel. error & 0.564\% & 0.294\% & 0.121\%   \\
						\hline
						\end{tabular}
						\label{table:categorisedInFSigma}
        } \\
        \subtable[]
        {
            \begin{tabular}{|c||c|c|c|}
						
						\hline
						$a^\nu_i$ & 0.1 & 1 & 10 \\
						\hline
						Mean rel. error & 0.727\% & 0.219\% & 0.033\%   \\
						\hline
						\end{tabular}
						\label{table:categorisedInFNu}
				} \\
				\subtable[]
        {
            \begin{tabular}{|c||c|c|c|c|c|}
						
						\hline
						$b^\si_j$ & (1, 1)& (1, 2)& (2,1)& (1, 5)& (5, 1) \\
						\hline
						Mean rel. error & 0.354\% & 0.275\% & 0.414\% & 0.149\% & 0.439\% \\
						\hline
						\end{tabular}
						\label{table:categorisedInMSigma}
        } \\
        \subtable[]
        {
            \begin{tabular}{|c||c|c|c|c|c|}
						
						\hline
						$b^\nu_j$ & (1, 1)& (1, 2)& (2,1)& (1, 5)& (5, 1) \\
						\hline
						Mean rel. error & 0.395\% & 0.344\% &0.143\% &0.212\% &0.537\%   \\
						\hline
						\end{tabular}
						\label{table:categorisedInMNu}
				}               
                
			\end{center}
			\caption{Mean relative error categorised in $\rho_1$ \subref{table:categorisedInRho}, the variables controlling the order of magnitude of $\si_i$ and $\nu_i$, namely $a^\si_i$ \subref{table:categorisedInFSigma} and $a^\nu_i$  \subref{table:categorisedInFNu}, and the variables controlling the imbalance, $b^\si_j$  \subref{table:categorisedInMSigma} and $b^\nu_j$ \subref{table:categorisedInMNu}.}\label{table:categorisedInParameterValues}
\end{table}

\begin{figure}
\begin{center}
\includegraphics[width=\figurewidth]{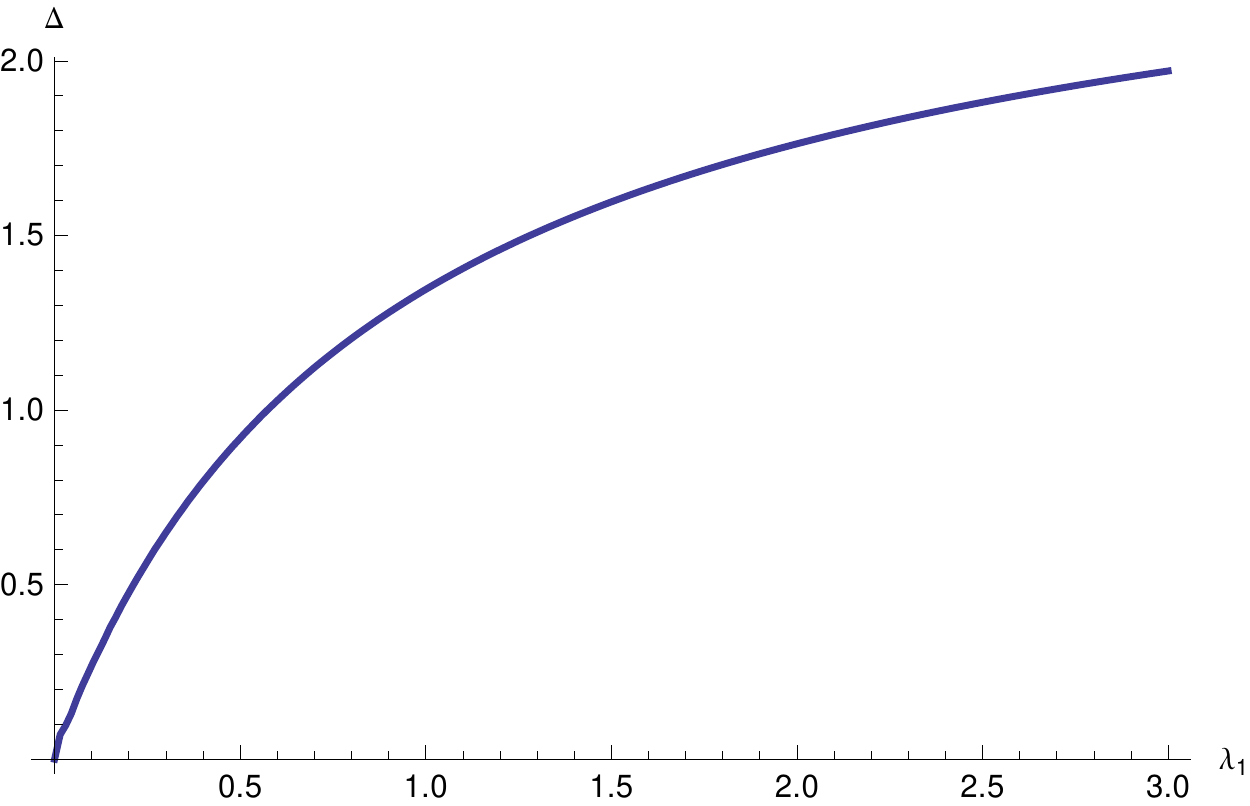}
\caption{The relative error made as a function of the products' arrival rate.}
\label{fig:numericsFastProducts}
\end{center}
\end{figure}

\paragraph*{Effect of fast moving products.}
From Table \ref{table:categorisedInParameterValues} we see that decreasing the uptimes and repair times of the machines relative to the movement speed of the products, or equivalently, that increasing the moving speed of products (i.e. arrival rate and service rate) relative to the uptimes and repair times, leads to a decrease in the performance of the approximation. To further examine this effect, we regard the length of $Q_1$ in systems with arrival rates ranging from $\la_1=0$ to $\la_1=3$ and an exponentially distributed service time $B_1$ with rate $10\la_1/3$ varying accordingly, so as to keep the workload fixed. Furthermore, the breakdown rates {are given by $\si_1=\si_2=1$ and the repair times $R_1$ and $R_2$ are exponentially (1) distributed.} After applying Approximation \ref{apprx:approx} on the mean queue length of $Q_1$ in these systems and comparing it with exact results, we obtain Figure \ref{fig:numericsFastProducts}, where the relative error $\Delta$ (see \eqref{eq:definitionRelError}) is given as a function of $\lambda_1$. We indeed observe that the faster the products arrive (and get served), the more inaccurate the approximation becomes. This effect can be explained by the fact that faster moving products are more sensitive to variations caused by dependence in the downtimes. A small increase in the downtime, causes more additional products to build up in the queue, while such an increase may even remain unnoticed in case of slow products with long interarrival times.  Hence, in the former, the error made in approximating the dependence structure of consecutive downtimes by the functions $\chi_1(\cdot)$ and $g_1(\cdot)$ shows itself more in the approximation of the mean queue length than in the latter.

\paragraph*{Effect of the degree of dependence.}
From Table \ref{table:categorisedInParameterValues} it is apparent that the accuracy of the approximation is influenced by the values for $b^\si_j$ and $b^\nu_j$. This can be mainly explained by the fact that these values determine the strength of the dependence between consecutive downtimes in $M_1$. 
To illustrate this effect, {let us observe systems where $B_1$, as well as both $R_1$ and $R_2$, are exponentially (1) distributed. Moreover, we have $\la_1 = 1/4$ and $\si_1=1$.} In Figure \ref{fig:numericsComparisonWithDependence}, we show the relative error $\Delta$ in approximating the mean length of $Q_1$ as a function of $\si_2$. Since the breakdown rate of $M_2$ varies in these systems, we have that the strength of the dependence changes accordingly. In Figure \ref{fig:numericsComparisonWithDependence}, $r_{scaled}$, the correlation coefficient of consecutive downtimes as computed in Section \ref{computationCorr} is given in a scaled form, so as to fit the graph. We see that the accuracy of the approximation is, at least in this case, largely determined by the strength of the correlation between the downtimes. Intuitively this makes sense, since in case there is no such correlation in the model (for example when $\si_2 = 0$ or $\si_2 \uparrow \infty$), the approximation should at least be close to being exact. Using the procedure of Section \ref{choosefunctions}, $g_1(\cdot)$ will resolve to zero in such a case, and when $\chi_1(\cdot)$ is chosen to match $\tildeD(\cdot)$, the approximation becomes exact, as the assumed downtime structure in \eqref{eq:combedecomposition}  with the functions $\chi_1(\cdot)$ and $g_1(\cdot)$ will then describe the dependence in an exact way.

\begin{figure}
\begin{center}
\includegraphics[width=\figurewidth]{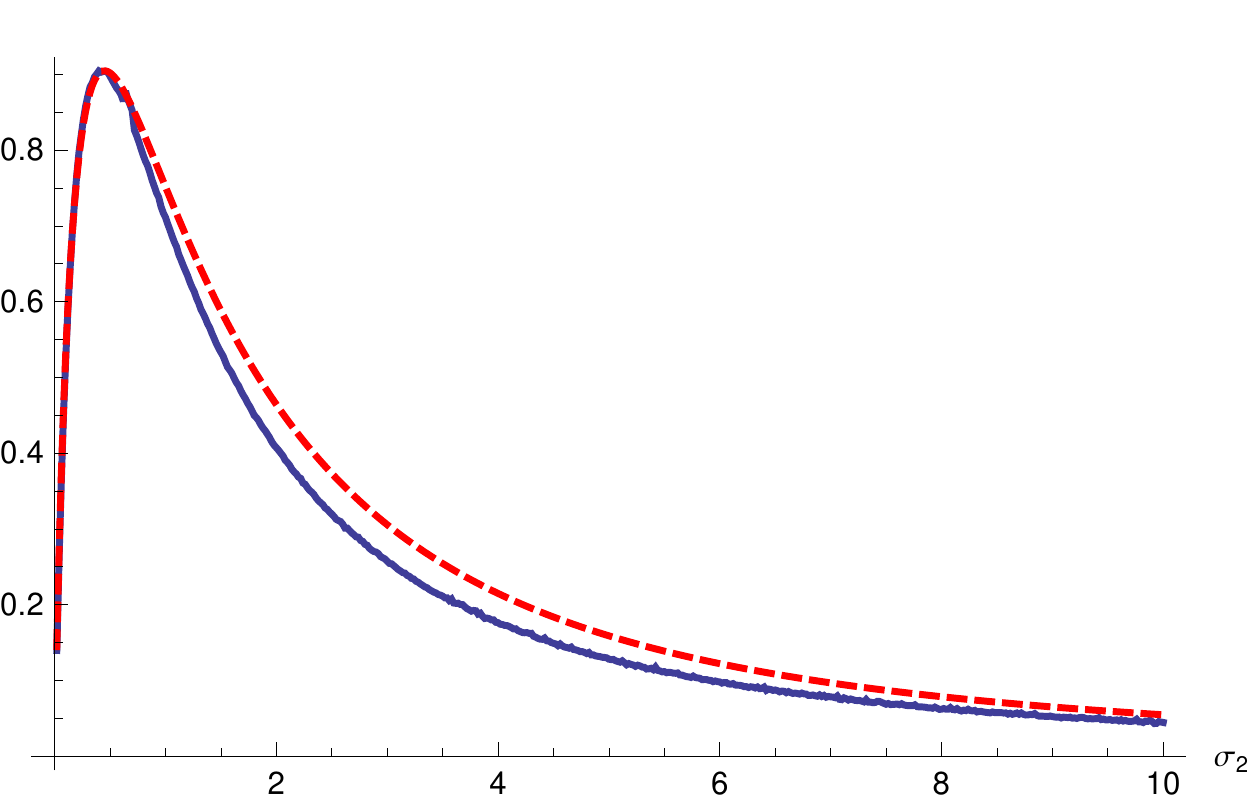}
\caption{The relative error made, $\Delta$ (continuous curve), and the scaled value of the correlation coefficient, $r_{scaled}$ (dashed curve) as a function of the breakdown rate of $M_2$.}
\label{fig:numericsComparisonWithDependence}
\end{center}
\end{figure}

\paragraph*{Effect of the variability of the repair times.}
In Table \ref{table:categorisedInParameterValues}, we have only included instances of the layered model for which repair times are exponentially distributed. In practice however, the level of variability in the repair times may be much higher. To investigate whether the accuracy of Approximation \ref{apprx:approx} is influenced by this, we again study the instance of the model as presented in Section \ref{initialGlance}. However, we now assume the repair times $R_1$ and $R_2$ to be identically, hyper-exponentially distributed with mean one. In particular, we study the behaviour of the relative error made by the approximation as the squared coefficients of variation of $R_1$ and $R_2$ ($\SCV{R_1}$ and $\SCV{R_2}$) increase.  Figure \ref{fig:varRepairTime} shows the relative error (as defined in \eqref{eq:definitionRelError}, however now with the sign included)  in approximating $\E{L_1}$ versus the squared coefficient of variation of the repair times. The various parameter combinations for the hyper-exponential repair-time distribution needed to satisfy the SCV values are chosen as described in \cite[p.\ 358--360]{Tijms}. The figure shows that even up to an SCV of 8, which represents highly variable repair times for both machines, the error made is only approximately 1\%. Therefore, the accuracy of Approximation \ref{apprx:approx} seems to remain very high even for very variable repair times.

\begin{figure}
\begin{center}
\includegraphics[width=\figurewidth]{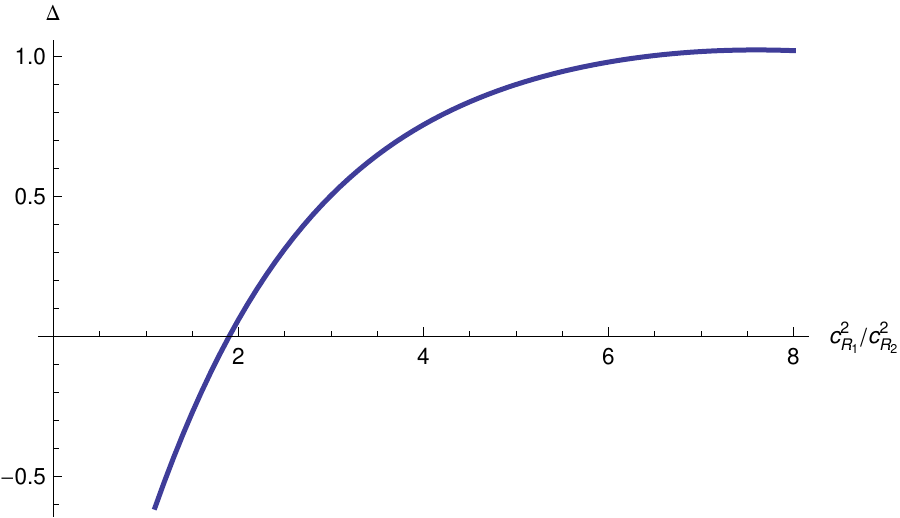}
\caption{The accuracy of the approximation as a function of the squared coefficient of variation of the repair times $R_1$ and $R_2$}
\label{fig:varRepairTime}
\end{center}
\end{figure}

\paragraph*{Comparison with Wartenhorst's approximation.} The approximation approach we used involves the study of the dependence in the second layer of the model. The two-layered model has been studied before by Wartenhorst \cite{Wartenhorst}. However, in \cite{Wartenhorst} it is assumed that $\si_1=\si_2$ {and that $R_1$ and $R_2$ are exponentially distributed with equal rates. This results in $\tildeD_1(\cdot) = \tildeD_2(\cdot)$.} In his study, Wartenhorst approximates the mean length of $Q_1$ with the mean queue length in a single server vacation queue, where the distribution of the vacation lengths equals the stationary downtime distribution of $M_i$, but where the downtimes are assumed to be \emph{completely independent}. The queue length distribution of this single server queue is obtained by applying the Fuhrmann-Cooper decomposition (cf. \cite{FuhrmannCooper}). Wartenhorst's approximation is exact by construction for a system where downtimes are independent, and accurate whenever downtimes are only slightly dependent. Although \cite{Wartenhorst} assumes equal breakdown rates and identically distributed repair times for the machines, with some effort his approach can be extended to allow for cases where these assumptions are violated. 

To compare the accuracy of the approximation derived in the present paper with that of \cite{Wartenhorst}, we study a set of systems with highly dependent downtimes. For these systems,  {let $\si_1 = $100, $\si_2 = 0.02$, and $R_2$ is exponentially distributed with rate $0.01$}. To maximise the correlation in the downtimes of $M_1$, we assume $R_1$ to be hyper-exponentially distributed with probability parameters 0.975 and 0.025, and rate parameters 100 and 0.01. The value for the correlation coefficient in these systems evaluates to 0.26. We vary $\la_1$ between 0 and 0.01. Furthermore, we assume $B_1$ to be exponentially distributed with rate 500$\la_1$, so as to keep the workload at $Q_1$ fixed. 

In Figure \ref{fig:usVsWartenHorst}, the relative error $\Delta$ in approximating $\E{L_1}$ is given for both the approximation obtained in this paper and Wartenhorst's approximation.
We see the same effect of fast moving products as before. The faster the products move, the less accurate both approximations become. However, we see that the degree of dependence has a significantly larger effect on the accuracy of Wartenhorst's approximation than on that of the approximation presented here. 
Since the degree of the dependence between the downtimes is the major source of inaccuracy for both approximations (cf. Section \ref{effectDependence}), one could conclude that the approximation we derived performs as well as Wartenhorst's approximation in cases with only slight dependences, and even better in cases with stronger correlations between the downtimes.

\begin{figure}
\begin{center}
\includegraphics[width=\figurewidth]{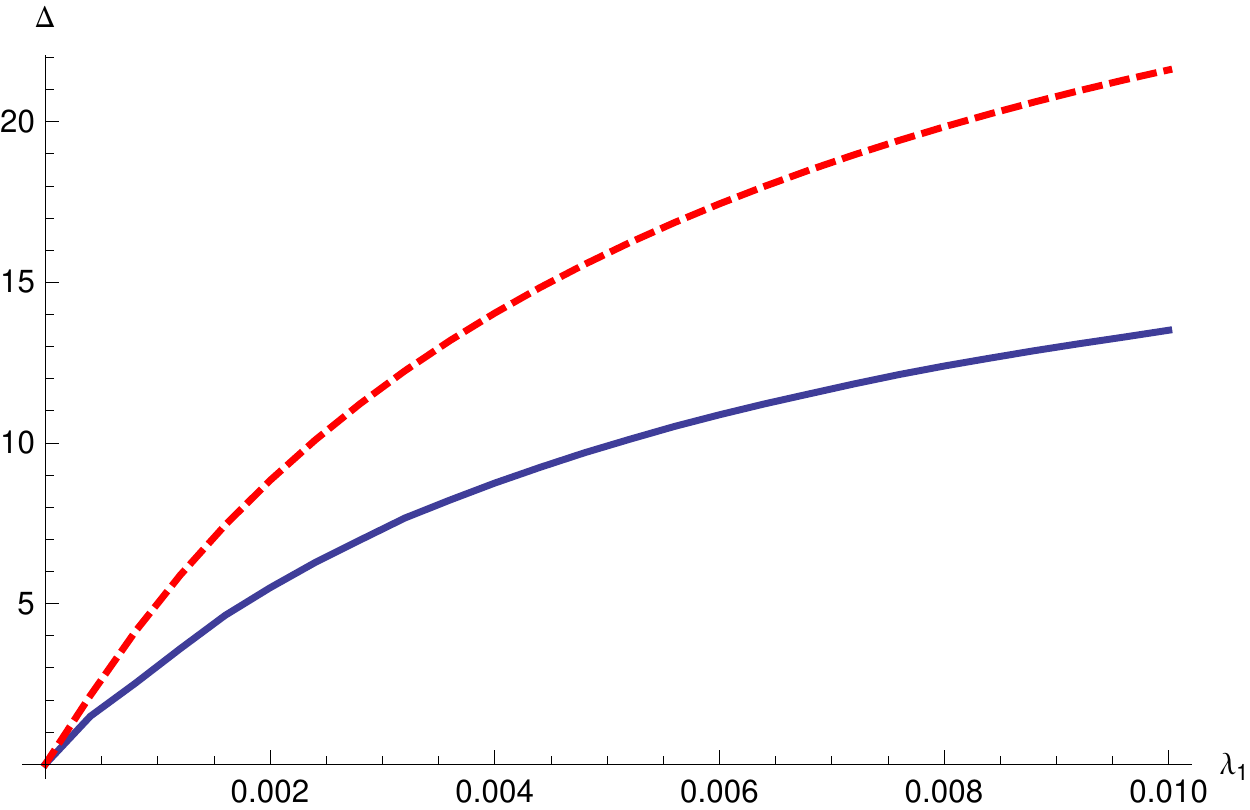}
\caption{The relative error made by the approximation in the present paper (solid curve) and Wartenhorst's approximation (dashed curve)}
\label{fig:usVsWartenHorst}
\end{center}
\end{figure}

\section{Conclusions}\label{conclusion}
We have considered a layered queueing network with two interacting layers, with machines being servers of products in one layer and customers in the other. Since the machines compete for repair facilities in the second layer, the queue lengths of products in the first layer are correlated, which makes analysis of this layer complicated. In fact, even when making symplifying assumptions such as Poisson product arrivals and exponential repair times for the machines, there are no explicit formulas available for the queue length distributions of the queues in the first layer. 

We have analysed the queue length distributions of these queues by noting that each of these queues in isolation can be seen as a single server queue with correlated server downtimes. What is not known, however, is how these downtimes are correlated exactly. Thus, in Section \ref{analysisSSQ}, we have derived an exact expression for the (PGF of the) queue length distribution for an M/G/1 single server queue with correlated successive server downtimes. This result seems to be of independent interest. The downtimes dependence was given by \eqref{eq:combedecomposition}, which covers a wide class of downtime dependence structures. As discussed earlier, one can view \eqref{eq:combedecomposition} as  a decomposition of the downtimes in an independent component and a dependent component.

The single server model, and in particular the specific dependence structure we considered, does not provide an exact representation of the dependence of successive downtimes of a machine in the layered model, but we were able to mimic that dependence very accurately by using information on the first two moments of the downtime distribution and the correlation coefficient of two consecutive downtimes of a machine. Then, the results on the single server queue given in Section \ref{analysisSSQ} form an approximation for the PGF of the queue length distribution of any first-layer queues. 

We have examined the accuracy of our approximation by performing a numerical study in Section~\ref{numStudy}. Although we have provided an approximation for the whole distribution of the marginal queue length in Approximation~\ref{apprx:approx}, in Section~\ref{numStudy} we provided some numerical results for the mean queue lengths. In general, we saw that the approximation works extremely well for a wide range of parameter settings. 

The major findings were as follows. First, for typical systems, our approximation is within 5\% of simulation results, while for more than half of the parameter settings we tested, the relative error was less than 0.1\%. Thus, the approximation is very accurate throughout the range of the parameters involved. 

Second,  we observed that the accuracy of the approximation is not very sensitive to the load of the queues while the orders of magnitude of the breakdown and repair rates do impact the accuracy of the approximation. However, the effect is due to the difference of time scales of product arrivals and services compared to that of machine breakdowns and repair times. Thus, we conclude that different time scales on the two layers are the major source of inaccuracy, while \eqref{eq:combedecomposition} works very well, despite the fact that it does not capture perfectly the dependence of the layered model; see also Remark~\ref{rem:reasonApprox}. The effect of the different time scales was also observed when considering fast moving products, where we observed that fast moving products are more sensitive to variations caused by the dependence in the downtimes. 

Additionally, we saw that variability of the repair times, as long as it does not affect significantly the correlation of the downtimes, does not have a big impact on the performance of our approximation. Last, we conclude that our approximation performs at least as well as the other existing approximation \cite{Wartenhorst} for this system, and in general it performs better. As a final conclusion, one cannot ignore the dependence or different time scales between layers in layered queueing networks.

%\bibliographystyle{plain}
%\bibliography{bib}

\appendix
\section{Proof of Lemma \ref{lem:uniqueRootPhi}}\label{proofLemmaDenominator}
\begin{proof}
The function $E(p)$ is continuous on $[0, \mu(\si))$. We also have that $1-E(0) = 1$ and $\lim_{p\rightarrow \mu(\si)} 1-E(p) = -\infty$. Hence, there exists at least one root in $(0,\mu(\si))$ by Bolzano's theorem.

To prove that there is at most one root in $(0,\mu(\si))$, we show that $1-E(p)$ is strictly decreasing in $p$, or equivalently, that $E(p)$ is strictly increasing in $p$ by studying the monotonicity of each of the terms in \eqref{eq:Ep} separately. 
First, since $\chi(\cdot)$ is the LST of a positive continuous random variable (see Section \ref{notation}), it is a strictly decreasing function. Recalling that $\lambda>0$, this means that the first term $\chi(\la(1-p))$ is therefore strictly increasing in $p$. 
For the monotonicity of the second term $A^2(p)$, we show that $A(p)$ is strictly decreasing, or equivalently, $A'(p) < 0$ for all values of $p$ considered. We have that 
\begin{align}
A'(p)  =& \frac{\si\la}{(\si+\la(1-p))^2}\frac{p(1-\tildeB(\si+\la(1-p)))}{p-\tildeB(\si+\la(1-p))} \notag \\
& + \frac{\si}{\si+\la(1-p)}
\Big(\frac{(1-\tildeB(\si+\la(1-p)))+p\la\tildeB'(\si+\la(1-p))}{p-\tildeB(\si+\la(1-p))} \notag \\ 
&\qquad \qquad \qquad \qquad \qquad-\frac{p(1-\tildeB(\si+\la(1-p)))(1+\la\tildeB'(\si+\la(1-p)))}{(p-\tildeB(\si+\la(1-p)))^2}\Big). \label{eq:APrime}
\end{align}
Since $\tildeB(\cdot)$ is the LST of a positive, continuous random variable, we have that $1-\tildeB(\si+\la(1-p))>0$ and $\tildeB'(\si+\la(1-p))>0$, which also readily implies that $p\la\tildeB'(\si+\la(1-p))>0$ and $1+\la\tildeB'(\si+\la(1-p))>0$. This means that in \eqref{eq:APrime}, the numerator of the second fraction in the first term and the numerators of the fractions between the brackets are all positive. Moreover, we have that $p-\tildeB(\si+\la(1-p))<0$ for all $p \in (0, \mu(\si))$, which consequently implies through the denominators that the second fraction of the first term and the expression between the brackets each are negative. Combining this with the fact that evidently both $\si/(\si+\la(1-p))$ and $\si\la/(\si+\la(1-p))^2$ are positive as $p<1$, we have that $A'(p)<0$ and thus that the second term $A^2(p)$ is strictly increasing. For the third term $\tildeD(\la(1-p)+g(\la(1-p)))$, we have that $\la(1-p)+g(\la(1-p))$ is strictly decreasing in $p$, for $g(s)$ is increasing in $s$, because $e^{-g(s)}$ is the LST of a positive continuous random variable and therefore strictly decreasing in $s$. Since $\la(1-p)+g(\la(1-p))$ is strictly decreasing in $p$ and the LST $\tildeD(\cdot)$ is a strictly decreasing function, the third term is strictly increasing in $p$.

All of the terms in \eqref{eq:Ep} are strictly increasing for the values of $p$ considered. As a result, $E(p)$ itself is strictly increasing for $p \in (0, \mu(\si))$. Therefore, the denominator $1-E(p)$ has exactly one root on the real line in $(0, \mu(\si))$.
\end{proof}


\section*{Supplementary material}

# A note on 'Marginal queue length approximations for a two-layered network with correlated queues'

J.L. Dorsman *†  
j.l.dorsman@tue.nl

O.J. Boxma *  
o.j.boxma@tue.nl

M. Vlasiou *†  
m.vlasiou@tue.nl

December 22, 2014

## Erratum

In [1], the following errors exist:

- In the lines "Under an explicit assumption...to approximate the machine-repair model" of the abstract, in the third-to-last and fourth-to last paragraph of Section 1, in the first paragraph of Section 3, in the paragraph directly above Remark 3.2 and in the third line of the second paragraph of Section 6, the word "exact" should be replaced by "approximate".

- In the second line of the second paragraph of Section 3, the word "approximate" should be added between "obtain" and "PGFs".

- In Section 4.2, the line "as they are the only source of error introduced" should be removed.

- The equality sign in (12), as well as the equality signs in the equalities in the remainder of Section 3 that are a result of (12), should be interpreted as an equality by approximation.

## Discussion

Contrary to the statements in the article that we compute the exact (probability generating function of the) queue length distribution of the single-server model, the expressions computed for $\mathbb{E}[p^L|\text{server up}]$ and $\mathbb{E}[p^L|\text{server down}]$ given in Lemmas 3.3 and 3.4, respectively, are only approximate and do not hold exactly. Therefore, Theorem 3.5 only leads to an approximation of the single-server model. This is due to the fact that in Section 3.1.2, the quantities $M(k)$ and $D(k)$ are assumed to be independent. This assumption does not hold true, since these both quantities both depend on the duration of $D(k-1)$, the downtime directly preceding the $k$-th uptime. As such, the two quantities are interdependent, so that (12) is only an approximation. This works through in later computations.

The main result of this paper, Approximation 5.4.1, is however not affected by this issue. The derived approximation as well as the observations made in Section 5.5 concerning its accuracy remain unchanged.

Computing the exact distribution of the single-server model does not seem straightforward. When we drop the approximation assumption of independence between $D(k)$ and $M(k)$, however, the analysis becomes considerably harder. To account for the dependence, one would have to condition throughout on the event $D(k-1) = s$ instead of the event $D(k) = t$. Expressions equivalent to (10), (12), (13) and (14) can still be obtained in the same fashion as before:

$$\mathbb{E}[p^{M(k)} \mid D(k-1) = s] = A(p)\mathbb{E}[p^{N(k)} \mid D(k-1) = s] + K(p)\mathbb{E}[\mu^{N(k)}(\sigma) \mid D(k-1) = s],$$
$$\mathbb{E}[p^{N(k+1)} \mid D(k-1) = s] = \mathbb{E}[p^{M(k)} \mid D(k-1) = s]\chi(\lambda(1-p))e^{-g(\lambda(1-p))s},$$
$$\mathbb{E}[p^{M(k+1)} \mid D(k-1) = s] = A(p)\mathbb{E}[p^{N(k+1)} \mid D(k-1) = s] + K(p)\mathbb{E}[\mu^{N(k+1)}(\sigma) \mid D(k-1) = s]$$

and

$$\mathbb{E}[p^{N(k+2)} \mid D(k-1) = s] = \mathbb{E}[p^{M(k+1)} \mid D(k-1) = s]\chi(\lambda(1-p))\chi(g(\lambda(1-p)))e^{-g(g(\lambda(1-p)))s}.$$

Concatenating these results leads to a relation of the following form:

$$\mathbb{E}[p^{N(k+2)} \mid D(k-1) = s] = E(p,s)\mathbb{E}[p^{N(k)} \mid D(k-1) = s] + F(p,s)\mathbb{E}[\mu^{N(k)}(\sigma) \mid D(k-1) = s]$$
$$+ G(p,s)\mathbb{E}[N(k)\mu^{N(k)}(\sigma) \mid D(k-1) = s].$$

Extracting a relation between $\mathbb{E}[p^{N(k+2)}]$ and $\mathbb{E}[p^{N(k)}]$ from this expression is not straightforward, as $p$ appears in both the coefficients and the expectations of the right-hand side of this equation.

## Acknowledgement



\end{document}